\documentclass{amsart}
\usepackage{algorithm}
\usepackage{algorithmic}
\usepackage{amsmath,amsthm}
\usepackage{verbatim}
\usepackage{marvosym}
\usepackage{graphicx}
\usepackage{color}
\usepackage{epsfig}
\usepackage{rotating}
\usepackage{lscape}
\usepackage{bm}
\usepackage{mathrsfs, euscript, mathdesign}
\usepackage[normalem]{ulem}
\usepackage{cleveref}
    \usepackage{baskervald}
\usepackage[T1]{fontenc}

\usepackage{tikz}
\usetikzlibrary{cd}
\usepackage{thmtools}
\usepackage{thm-restate}

\theoremstyle{plain}
\newtheorem{theorem}{Theorem}[section]
\newtheorem{cor}[theorem]{Corollary}

\newtheorem{prop}[theorem]{Proposition}

\newtheorem{example}[theorem]{Example}

\theoremstyle{definition}

\newtheorem{remark}[theorem]{Remark}

{\kern -4pt}

\makeatletter
\newcommand*{\house}[1]{%
  \mathord{%
    \mathpalette\@house{#1}%
  }%
}
\newcommand*{\@house}[2]{%
  \dimen@=\fontdimen8 %
      \ifx#1\scriptscriptstyle\scriptscriptfont
      \else\ifx#1\scriptstyle\scriptfont
      \else\textfont\fi\fi
      3 %
  \sbox0{%
    $#1%
      \vrule width\dimen@\relax
      \overline{%
        \kern2\dimen@
        \begingroup 
          #2%
        \endgroup
        \kern2\dimen@
      }%
      \vrule width\dimen@\relax
      \mathsurround=1.5\dimen@ 
    $%
  }%
  \ht0=\dimexpr\ht0-\dimen@\relax
  \dp0=\dimexpr\dp0+2\dimen@\relax
  \vbox{%
    \kern\dimen@ 
    \copy0 
  }%
}

\newcommand{\Ox}{\ensuremath{\mathcal{O}}}

\newcommand{\Q}{\ensuremath{\mathbb{Q}}}
\newcommand{\R}{\ensuremath{\mathbb{R}}}
\newcommand{\Z}{\ensuremath{\mathbb{Z}}}
\newcommand{\C}{\ensuremath{\mathbb{C}}}

\newcommand{\kate}[1]{{\color{blue}{#1}}}

\newcommand{\bd}{2\cdot 10^7}

\newcounter{nootje}
\begin{document}

\title{Minimal Mahler Measure in Quartic Galois Number Fields}

\author{ Bishnu Paudel, Kathleen Petersen and Haiyang Wang}

\thanks{The second author was supported by an AMS-Simons PUI Grant}



\begin{abstract}
We explore the dependence of the minimal integral Mahler measure of Galois quartic fields on the discriminant of the field.  We obtain density results which are conditional on the ABC conjecture as well as several unconditional results.

 \end{abstract}

\maketitle

\section{Introduction}

For a number field $K$ we let $D_K$ be the absolute discriminant of $K$ and $\Ox_K$ the ring of integers of $K$. 
The Mahler measure of  a non-constant  polynomial $f(x)=c \prod_{i=1}^d (x-\alpha_i)\in \C[x]$ is  
\[
M(f) = |c| \prod_{ |\alpha_i|\geq 1} |\alpha_i|,
\]
and for an algebraic number $\alpha$, we define $M(\alpha)$ to be the Mahler measure of a minimal polynomial for $\alpha$ over $\Z$ (with content 1).  
The minimal Mahler measure of a number field $K$ is the minimal Mahler measure of a generator
\[
M(K) = \min\{ M(\alpha): \Q(\alpha) = K \},
\]
and $M(\Ox_K)$ is the minimal Mahler measure of an integral generator. 
We study the dependence of  $M(\Ox_K)$ on  the discriminant $D_K$ for Galois quartic number fields $K$, focusing on the exponent of the discriminant in this dependence. 
The main tools we use are Liouville's theorem in Diophantine approximation and Granville's work on square-free values of polynomials.  Granville's work depends on the ABC conjecture in many cases and we use these square-free values to construct explicit number fields. The known bounds for $M(\Ox_K)$ specialize to the following for Galois quartics (see Section~\ref{section:background}). If $K$ is a totally imaginary quartic and $\text{Tor}(K^{\times})\neq \{\pm 1\}$ then   
\begin{equation}\tag{$*$}
2^{-\frac{12}5} D_K^{\frac15} \leq M(\Ox_K) \leq  \big(\tfrac{2}{\pi}\big)^{2} D_K^{\frac1{2}}
\end{equation} 
and if  $\text{Tor}(K^{\times})=\{\pm 1\}$ then 
\begin{equation}\tag{$**$}
2^{-\frac{12}5} D_K^{\frac15} \leq M(\Ox_K) \leq D_K.
\end{equation} 
If $K$ is a  totally real quartic field then 
\begin{equation}\tag{$***$}
2^{-\frac43} D_K^{\frac16}  \leq M(\Ox_K) \leq  D_K^{\frac12}.
\end{equation}

We now summarize our main results, which differ depending on the Galois group of $K$ and whether $K$ is totally imaginary or totally real. 
For totally real biquadratic fields, assuming the ABC conjecture, we show that  all possible rational exponents are realized for infinitely many fields. 
\begin{theorem}\label{allexpreal}
Let $\tfrac16 \le \tfrac{p}{q} \le \tfrac12$ be a rational number. There are absolute  constants $c_1, c_2>0$ such that assuming the ABC conjecture there are infinitely many totally real biquadratic fields $K$ for which 
\[ c_1 D_K^{\frac{p}{q}} \le M(\mathcal{O}_K) \le c_2 D_K^{\frac{p}{q}}.\]
\end{theorem}
\noindent
In Corollaries~\ref{cor:realbiqadratic16}, \ref{cor:realbiqadratic14}, and \ref{cor:realbiqadratic12} we  produce explicit unconditional examples for exponents $\tfrac{p}{q} = \tfrac{1}{6}, \tfrac{1}{4}$ and $\tfrac{1}2$.

In Theorem~\ref{real>14} we improve the theoretical lower bound for totally imaginary biquadratics by increasing the exponent from $\tfrac15$ to $\tfrac14$ .  Assuming the ABC conjecture, we show that all possible rational exponents are realized for infinitely many fields. 
\begin{theorem}\label{thm:conditionalimaginary}
Let $\tfrac14 \le \tfrac{p}{q} \le 1$ be a rational number. There are absolute  constants $c_1, c_2>0$ such that assuming the ABC conjecture there are infinitely totally imaginary biquadratic fields $K$ for which 
\[ c_1 D_K^{\frac{p}{q}} \le M(\mathcal{O}_K) \le c_2 D_K^{\frac{p}{q}}.\]
\end{theorem}
\noindent
In Corollaries~\ref{cor:imagbiqadratic14}, \ref{cor:imagbiqadratic12}, \ref{cor:imagbiqadratic23}, and \ref{cor:imagbiqadratic1} we produce explicit unconditional examples for exponents $\tfrac{p}{q} =\tfrac{1}{4}, \tfrac{1}{2}, \tfrac{2}{3}$ and $1$.

For totally real cyclic quartic fields, in Theorem~\ref{theorem:realcyclic16},  we show that the theoretical lower bound exponent of $\tfrac16$ is sharp. 
The range of exponents in the theoretical bounds is $[\tfrac16, \tfrac12]$.  Assuming the ABC conjecture, we show that all rational exponents in the range $[\tfrac1{4},\tfrac12)$ are realized by infinitely many fields.
\begin{theorem}\label{realcyc_main}
	Let $\tfrac{1}{4} \le \tfrac{p}{q} < \tfrac12$ be a rational number. There are absolute  constants $c_1, c_2>0$ such that assuming the ABC conjecture there are infinitely many totally real cyclic quartic fields $K$ for which 
	\[ c_1 D_K^{\frac{p}{q}} \le M(\mathcal{O}_K) \le c_2 D_K^{\frac{p}{q}}.\]
\end{theorem}
\noindent
In Theorem~\ref{theorem:realcyclic12} we produce an infinite family of examples realizing the exponent of $\tfrac12$.
Our work leaves open the question of whether there is a similar density property for real cyclic quartics and exponents in the range $(\tfrac16, \tfrac1{4})$. We also calculate the minimal Mahler measure of all totally real cyclic quartic fields of discriminant at most $\bd$, and present these results in Figure~\ref{fig.realcyc1} and Figure~\ref{fig.realcyc2}.

In Theorem~\ref{thm_cyccpxlb} we improve the theoretical lower bound for totally imaginary cyclic quartic fields by increasing the exponent $\tfrac15$ to $\tfrac13$. Assuming the ABC conjecture, we show that all possible rational exponents are realized for infinitely many fields. 
\begin{theorem}\label{thm:condimg_cyc}
	Let $\tfrac13 \le \tfrac{p}{q} \le 1$ be a rational number.   There are absolute  constants $c_1, c_2>0$ such that assuming the ABC conjecture there are infinitely many totally imaginary cyclic quartic fields $K$ for which
	\[
	c_1\,D_K^{\frac{p}{q}}
	\;\le\;
	M\bigl(\mathcal{O}_K\bigr)
	\;\le\;
	c_2\,D_K^{\frac{p}{q}}.
	\]
\end{theorem}
\noindent
In Corollaries~\ref{cor:ccyclic13}, \ref{cor:ccyclic12},  and \ref{cor:ccyclic1} we  produce explicit unconditional examples for exponents $\tfrac{p}{q} = \tfrac{1}{3}, \tfrac{1}{2}$ and $1$.


\subsection{Background}\label{section:background}

In 1964, Mahler \cite{MR166188} related what we now call the Mahler measure to the discriminant, and this work implies that for $d=[K:\Q]\geq 2$, 
\[
d^{-\frac{d}{2d-2}} |D_K|^{\frac1{2(d-1)}} \leq M(\Ox_K).
\]
In 1984, Silverman  \cite[Theorem 2]{MR747871}  extended this result to  the non-integral setting, proving
\[ 
	d^{-\frac{d}{2d-1}}|D_K|^{\frac1{2(d-1)}}\leq  M(K). 
\]
Ruppert \cite[page 18]{MR1624340} and Masser \cite[Proposition 1]{MR1363143} showed that the lower bound is sharp by  providing a family of fields for which this lower bound  cannot be improved by a larger exponent on the discriminant. (See also \cite{MR3413883}.)  Vaaler and Widmer  \cite{MR3413883} showed that  for composite $d$ there are number fields $K$ for which no constant $c_d$ satisfies
\[
M(K)  \leq c_d |D_K|^{\frac1{2(d-1)}} 
\]
demonstrating that there are fields whose minimal Mahler measure grows faster than Silverman's lower bound.
Child and Widmer \cite[Corollary 1]{MR4742655} improved the lower bound for totally imaginary number fields in the integral case to 
\[
2^{\frac{d(1-d)}{2d-3}} |D_K|^\frac1{2d-3} \leq M(\Ox_K). 
\]

The best general upper bound for $M(\Ox_K)$ is 
\[ M(\Ox_K)\leq |D_K|\] which follows from Minkowski's convex body theorem (see \cite[Lemma 7.1]{MR4208221}). 
Ruppert \cite[Proposition 3]{MR1624340}  proved 
 that if $K$ is totally real of prime degree then $ M(\Ox_K)\leq    |D_K|^{\frac12}.$
(This argument  can be extended to all number fields of prime degree using an extension of Minkowski's linear forms theorem.)
Vaaler and Widmer  \cite[Theorem 1.2]{MR3074815}  proved that if $K$ is not totally imaginary, and $r_2$ denotes the number of complex places of $K$,   with $c_K =  \big(\tfrac{2}{\pi}\big)^{r_2} |D_K|^{\frac1{2}}$, then \[ M(\Ox_K) \leq c_K.\]
(The theorem does not state explicitly that the primitive element is integral, and the integrality follows from the proof. See also the commentary in \cite{AVW1}.)
Vaaler and Widmer  show \cite[Theorem 1.3]{MR3074815}  that if $K$ is totally imaginary there is a constant $B=B(d)$ such that  \[ M(K) \leq B |D_K|^{\frac12}\]  under the assumption of the truth of the generalized Riemann hypothesis for the Galois closure of $K$. This result produces a non-integral small height generator. Akhtari, Vaaler and Widmer \cite[Corollary 1.1]{AVW1} have shown that  \[ M(\Ox_K) \leq c_K\]   for $K$ with $\text{Tor}(K^{\times}) \neq \{\pm 1\}$.  Conversely, they show  \cite[Theorem 1.3]{AVW1} that if $c_K < M(\Ox_K)$ then $K$ is totally imaginary and has a totally real subfield $F$ such that $K/F$ is Galois and $\mathrm{Tor}(K^{\times})=\{ \pm 1\}$. Further, if $E\subset K$ is totally real, then $E\subset F$. They provide the example (in Section 1) of $K=\Q(\sqrt{\sqrt{3}-2})$ with these features but where the inequality $c_K< M(\Ox_K)$ holds.  For any even degree $d$ (at least 4) they also construct infinitely many CM fields of degree $d$ that satisfy $c_K < M(\Ox_K)$.
In Section~\ref{section:rootsunity}, for all Galois quartics $K$ containing non-trivial roots of unity (roots of unity other than $\pm 1$),  we determine explicit elements in $\Ox_K$ with Mahler measure smaller than $c_K$. Theorem~\ref{thm:conditionalimaginary} and \ref{thm:condimg_cyc} demonstrate that there are infinitely many totally imaginary biquadratic fields and infinitely many totally imaginary cyclic quartic fields $K$, all of which are CM, such that $c_K<M(\Ox_K)$ as well as infinitely many of both types of fields for which $\text{Tor}(K^{\times})=\{\pm 1\}$ such that $M(\Ox_K)\leq c_K.$

Many of the results  above stated using a different height. Our work could equivalently be stated, for example, in terms of the absolute multiplicative Weil height of an algebraic number $\alpha$, $H(\alpha)$, using the fact that if the degree of $\alpha$ is $d$ then $H(\alpha) = M(\alpha)^{\frac1d}$.  We will discuss number fields $K$ in terms of explicit embeddings into $\R$ or $\C$ as the Mahler measure of two conjugate elements is equal.

\subsection{Results in small degree}

For quadratic number fields, the upper and lower bounds have the same exponent on the  discriminant.  Cochrane et al \cite{MR3463562} showed that for real quadratics 
\[ \tfrac12 |D_K|^{\frac12} \leq M(K) \leq |D_K|^{\frac12}. \]    
For cubic number fields  the aforementioned bounds specialize to 
\[ 3^{-\frac{3}{4}}|D_K|^{\frac1{4}}\leq  M(\Ox_K) \leq  |D_K|^{\frac12}. \]
Eldredge and Petersen \cite{MR4468151},  showed that  the exponent on the lower bound is sharp and that there are infinitely many Kummerian cubic fields with exponent $\tfrac13$.

Galois quartic number fields are either totally real or totally imaginary and have positive discriminants.  The corresponding $c_K$ values  are $D_K^{\frac12}$ and $ \big(\tfrac{2}{\pi}\big)^{2} D_K^{\frac1{2}}$, respectively. This gives us the bounds $(*)$, $(**)$, and $(***)$.
Totally imaginary Galois quartic fields are CM fields. The only roots of unity other than $\pm 1$ that can be contained in a biquadratic are powers of $\sqrt{-1}$ or $e^{\frac{2\pi i}6}$, which occur only when $ \sqrt{-1}$ or $\sqrt{-3}$ are contained in $K$. We note that $\Q(\sqrt{-1}, \sqrt{-3})$ is a biquadratic field and is the splitting field of the cyclotomic polynomial $\Phi_{12}(x)=x^4-x^2+1$.  The splitting field of $\Phi_8(x)= x^4+1$, $\Q(\sqrt{-1}, \sqrt{2})$  is the only other cyclotomic biquadratic field.  For totally imaginary cyclic quartic fields, only the splitting fields of the cyclotomic polynomials $\Phi_5= x^4+x^2+x+1$ and $\Phi_{10}=x^4-x^3+x^2-x+1$  contain non-trivial roots of unity.

\subsection{Liouville's Theorem and Square-Free Values of Polynomials}

Many of our results use Liouville's  theorem in Diophantine approximation, which we state with an explicit approximation constant, see \cite[Chap. 6, Sec.1]{Cassels}.

\begin{theorem}[Liouville]\label{thm.Liou} 
Let $\alpha$ be a real algebraic number of degree $d \geq 2$. There is a constant $\mu=\mu(\alpha)>0$ such that for all rational numbers $\tfrac{p}{q}$, we have \[\left|\frac{p}{q}-\alpha\right|\ge \frac{\mu}{q^d}.\]
We can take 
\[
\mu(\alpha)=\frac{1}{a_d\prod_{i=2}^d(1+|\alpha|+|\alpha_i|)},\]
where $\alpha_i$ are the conjugates of $\alpha_1:=\alpha$ and $a_d>0$ is the leading coefficient of a minimal polynomial of $\alpha$ over $\mathbb{Z}$. 
\end{theorem}

We will also rely upon Granville's work \cite[Theorem 1]{MR1654759} which guarantees square-free values of polynomial functions. 

\begin{theorem}[Granville]\label{thm:Granville}
Assume that $f(x)\in \Z[x]$ has no repeated roots. Let $B$ be the largest integer that divides $f(n)$ for all $n$,  $B'$ be the smallest divisor of $B$ such that $B/B'$ is square-free,  $q_p$ be the largest power of $p$ dividing $B'$ and $\omega_f(p)$ be the number of integers a with $1\leq a \leq p^{2+q_p}$ satisfying $f(a)/B' \equiv 0 \pmod {p^2}$.  Define 
\[
c_f = \prod_{p \text{ prime}} \Big( 1- \frac{\omega_f(p)}{p^{2-q_p}} \big). 
\]
Assuming the truth of the ABC conjecture then  $c_f>0$ and 
\[
\# \{ n\leq N : f(n)/B' \text{ is square-free } \} \sim c_fN.
\]

\end{theorem}
As Granville notes,  this  result can be proven unconditionally if $f$ has degree $\leq 2$ using the sieve of
Eratosthenes. The result was proven unconditionally by Hooley \cite{MR214556} for $f$ of degree three.

\section{Discriminants and Integers of Quartic fields}\label{section:introquart}

Quartic fields can be distinguished by the number of their real and imaginary embeddings as well as their Galois groups. These Galois groups are either $S_4$, $A_4$, $D_8$ (the dihedral group with 8 elements), $\Z/2\Z\times \Z/2\Z$ or $\Z/4\Z$.  In this manuscript, we focus on quartic Galois number fields, those with Galois groups  $\Z/2\Z\times \Z/2\Z$ or $\Z/4\Z$.

The theoretical lower bounds differ depending on whether $K$ has a real embedding or not. The upper bounds are more subtle, notably by   \cite[Theorem 1.3]{AVW1}, in the quartic case, if $c_K=\big(\tfrac2{\pi} \big)^2 |D_K|^{\frac12} \leq M(\Ox_K)$ then $K$ is a CM field with no roots of unity other than $\pm 1$.  A number field $K$ is a CM-field if $K$ has only complex embeddings and there exists a totally real subfield $F\subset K$  such that $K/F$ is a quadratic extension.  Of note, all  totally imaginary cyclic and biquadratic quartics are  CM.

Let $\alpha \in K$. For later use, we define
\begin{equation}\label{eq.M'}
	M'(\alpha) = |c(\alpha)|^{-1} M(\alpha),
\end{equation}
where $c(\alpha)$ denotes the leading coefficient of the minimal polynomial of $\alpha$ over $\mathbb{Q}$ with content 1.  
Note that $M'(\alpha) = M(\alpha)$ when $\alpha \in \mathcal{O}_K$.

We now collect information about the discriminants and integer rings of Galois quartic number fields. 

\subsection{Biquadratic Fields}\label{section:introbiquad}

A biquadratic number field is a quartic Galois number field with Galois group isomorphic to the Klein 4-group, $\mathbb Z/2\mathbb Z \times \mathbb Z/2\mathbb Z$. These fields can be written as $\Q(\sqrt{m l}, \sqrt{nl})$, where $m$, $n$, and $l$  are pairwise coprime square-free  integers.  We can reduce to considering the cases when $(ml,nl)\equiv (1,1), (1,2), (2,3)$, or $(3,3)\pmod 4$ since 
\[
\Q(\sqrt{ml},\sqrt{nl}) = \Q(\sqrt{nl}, \sqrt{ml}) = \Q(\sqrt{ml}, \sqrt{mn}) = \Q(\sqrt{nl}, \sqrt{mn}).
\]
Because of these different representations of $K$, we will often assume that $|l| \leq |m|\leq |n|$.
  From \cite[Theorem 2]{MR279069} the  discriminant of $K=\Q(\sqrt{ml}, \sqrt{nl})$ satisfies $D_K=c (l m n)^2$ where $c$ is given by the following congruence modulo 4
\[
c = 
\left\{
\begin{aligned}
1  &  \ \text{ if } \ (ml,nl) \equiv (1,1)&\pmod 4, \\
16 & \ \text{ if } \ (ml,nl) \equiv (1,2) \text{ or } (3,3)&\pmod 4,  \\
64 & \ \text{ if } \ (ml,nl) \equiv (2,3)&\pmod 4. 
\end{aligned}
\right. 
\]
The ring of integers $\mathcal{O}_K$ of $K$ is also computed in  \cite[Theorem 1]{MR279069}. Elements in $\Ox_K$ have the form 
\[ f(x_0+x_1\sqrt{ml}+x_2\sqrt{nl} +x_3\sqrt{mn}),\] where $x_0, x_1, x_2, x_3$ are rational integers and 
\begin{itemize}
\item[$\circ$] if $(ml,nl)\equiv (m,n) \equiv (1,1) \pmod 4$ then $f=\tfrac14$, $x_0\equiv x_1 \equiv x_2 \equiv x_3 \pmod 2$ and $x_0-x_1+x_2-x_3\equiv 0 \pmod 4$, 
\item[$\circ$]  if $(ml,nl)\equiv (1,1),  (m,n) \equiv (3,3) \pmod 4$ then $f=\tfrac14$, $x_0\equiv x_1 \equiv x_2 \equiv x_3 \pmod 2$ and $x_0-x_1-x_2-x_3\equiv 0 \pmod 4$, 
\item[$\circ$]  if $(ml,nl)\equiv (1,2)  \pmod 4$ then $f=\tfrac12$, $x_0\equiv x_1, x_2 \equiv  x_3 \pmod 2$, 
\item[$\circ$]  if $(ml,nl)\equiv (2,3)  \pmod 4$ then $f=\tfrac12$, $x_0\equiv x_2 \equiv 0,  x_1 \equiv  x_3 \pmod 2$, 
\item[$\circ$]  if $(ml,nl)\equiv (3,3)  \pmod 4$ then $f=\tfrac12$, $x_0\equiv x_3,  x_1 \equiv  x_2 \pmod 2$. 
\end{itemize}

\subsection{Cyclic Quartic Fields}\label{section:introcyclic}

A cyclic quartic field is a quartic Galois number field with Galois group isomorphic to $\Z/4\Z.$ For such a $K$,  there exist unique integers $A,B,C,D$ such that
	\[ K \;=\; \Q\bigl(\sqrt{A\,(D + B\sqrt{D})}\bigr), \]	
	where $A$  is odd and square-free, $D = B^2 + C^2$  is square-free with $B,C>0$, and  $\gcd(A,D)=1$ (see \cite[Theorem 1]{zbMATH03995823}).
Any field satisfying these properties is a  cyclic quartic extension, and it is totally real if and only if $A > 0$.

The discriminant of $K$ (see \cite{zbMATH03995823}) is
\begin{equation}\label{eq.disc_cyc}
	D_K = c\,A^2\,D^3,
\end{equation}
where
\[
c =
\begin{cases}
	256 & \text{if }D\equiv0\pmod{2},\\
	64 & \text{if }D\equiv1\pmod{2},\;B\equiv1\pmod{2},\\
	16 & \text{if }D\equiv1\pmod{2},\;B\equiv0\pmod{2},\;A+B\equiv3\pmod{4},\\
	1 & \text{if }D\equiv1\pmod{2},\;B\equiv0\pmod{2},\;A+B\equiv1\pmod{4}.
\end{cases}
\]

Set
\begin{equation}\label{rho_sig}
	\rho = \sqrt{A\,(D + B\sqrt{D})}, 
	\qquad
	\sigma = \sqrt{A\,(D - B\sqrt{D})}.
\end{equation}
Then an integral basis of $\mathcal{O}_K$ is given by one of the following, according to congruence conditions on $A$, $B$, and $D$  (see \cite[Theorem on page 146]{MR1057983}):

\begin{itemize}
	\item[$\circ$] If $D\equiv0\pmod{2}$ then  is $\{\,1,\;\sqrt{D},\;\rho,\;\sigma\}$ an integral basis.
	\item[$\circ$]  If $D\equiv B\equiv1\pmod{2}$ then
	$
	\bigl\{\,1,\;\tfrac12(1+\sqrt{D}),\;\rho,\;\sigma\bigr\}
	$ is an integral basis.
	\item[$\circ$]  If $D\equiv1\pmod{2}$, $B\equiv0\pmod{2}$, and $A+B\equiv3\pmod{4}$ then 
	\[
	\bigl\{\,1,\;\tfrac12(1+\sqrt{D}),\;\tfrac12(\rho+\sigma),\;\tfrac12(\rho-\sigma)\bigr\}
	\]
	is an integral basis.
	\item[$\circ$] If $D\equiv1\pmod{2}$, $B\equiv0\pmod{2}$, $A+B\equiv1\pmod{4}$, and $A\equiv C\pmod{4}$ then
	\[
	\bigl\{\,1,\;\tfrac12(1+\sqrt{D}),\;\tfrac14\bigl(1+\sqrt{D}+\rho+\sigma\bigr),\;\tfrac14\bigl(1-\sqrt{D}+\rho-\sigma\bigr)\bigr\}
	\]
	is an integral basis.
	\item[$\circ$]  If $D\equiv1\pmod{2}$, $B\equiv0\pmod{2}$, $A+B\equiv1\pmod{4}$, and $A\equiv -C\pmod{4}$ then
	\[
	\bigl\{\,1,\;\tfrac12(1+\sqrt{D}),\;\tfrac14\bigl(1+\sqrt{D}+\rho-\sigma\bigr),\;\tfrac14\bigl(1-\sqrt{D}+\rho+\sigma\bigr)\bigr\}
	\]
	is an integral basis.
\end{itemize}

The proposition below follows directly from this statement. 

\begin{prop}\label{prop.cycint}
	Let $A, B, C, D \in \mathbb{Z}$ and the field $K$ be as above, and let $\rho, \sigma$ be as defined in equation~\eqref{rho_sig}.  
	Then the integers of the field $K$ are precisely the elements of the following forms where  $x_1, x_2, x_3, x_4 \in \mathbb{Z}$. :
	\begin{itemize}
		\item[$\circ$] If $D \equiv 0 \pmod{2}$, the integers are
		\[
		x_1 + x_2\sqrt{D} + x_3\rho + x_4\sigma.
		\]
		
		\item[$\circ$] If $D \equiv B \equiv 1 \pmod{2}$, the integers are
		\[
		\tfrac{1}{2}(x_1 + x_2\sqrt{D} + x_3\rho + x_4\sigma),
		\]
		where $x_1 \equiv x_2$, $x_3 \equiv x_4 \equiv 0 \pmod{2}$.
		
		\item[$\circ$] If $D \equiv 1 \pmod{2}$, $B \equiv 0 \pmod{2}$, and $A + B \equiv 3 \pmod{4}$, the integers are
		\[
		\tfrac{1}{2}(x_1 + x_2\sqrt{D} + x_3\rho + x_4\sigma),
		\]
		where $x_1 \equiv x_2$, $x_3 \equiv x_4 \pmod{2}$.
		
		\item[$\circ$] If $D \equiv 1 \pmod{2}$, $B \equiv 0 \pmod{2}$, $A + B \equiv 1 \pmod{4}$, and $A \equiv C \pmod{4}$, the integers are
		\[
		\tfrac{1}{4}(x_1 + x_2\sqrt{D} + x_3\rho + x_4\sigma),
		\]
		where $x_1 \equiv x_2 \equiv x_3 \equiv x_4 \pmod{2}$, $x_1 - x_2 - x_3 + x_4 \equiv 0 \pmod{4}$.
		
		\item[$\circ$] If $D \equiv 1 \pmod{2}$, $B \equiv 0 \pmod{2}$, $A + B \equiv 1 \pmod{4}$, and $A \equiv -C \pmod{4}$, the integers are
		\[
		\tfrac{1}{4}(x_1 + x_2\sqrt{D} + x_3\rho + x_4\sigma),
		\]
		where $x_1 \equiv x_2 \equiv x_3 \equiv x_4 \pmod{2}$, $x_1 - x_2 - x_3 - x_4 \equiv 0 \pmod{4}$.
	\end{itemize}
\end{prop}

\section{Real Biquadratics}\label{section:realbiquadratics}

Let $K=\Q(\sqrt{ml},\sqrt{nl})$ as described in Section~\ref{section:introbiquad}. We assume that $l < m< n$.
The discriminant satisfies $l^2m^2n^2\leq D_K\leq 64l^2m^2n^2$. The theoretical bounds  are given in $(***)$ and the exponents on the discriminant are $\tfrac16$ and $\tfrac12$. 
We first determine a lower bound for $M(\Ox_K)$ that is a constant times $n$. We use this to show that many real biquadratics  satisfy a lower bound with exponent on the discriminant of $\tfrac14$, improving the known bound of $\tfrac16$.  
We then show that assuming the ABC conjecture,  for any $\frac{p}{q}$ satisfying $\tfrac16 \leq \tfrac{p}{q} \leq \tfrac14$,  there are infinitely many real quadratics with $M(\Ox_K)$ a constant times $D_K^{\frac{p}{q}}$.

From Section~\ref{section:introbiquad}, an element  $\alpha=\alpha_1$ in $\Ox_K$ can be expressed as \[\alpha_1=\tfrac{1}{4}(a+b\sqrt{ml}+c\sqrt{nl}+d\sqrt{mn})\] for some $a,b,c,d\in\mathbb{Z}$. The other conjugates of $\alpha_1$ are 
\begin{align*}
\alpha_2&=\tfrac{1}{4}(a-b\sqrt{ml}+c\sqrt{nl}-d\sqrt{mn}), \\
\alpha_3&=\tfrac{1}{4}(a+b\sqrt{ml}-c\sqrt{nl}-d\sqrt{mn}),\\
\alpha_4&=\tfrac{1}{4}(a-b\sqrt{ml}-c\sqrt{nl}+d\sqrt{mn}).
\end{align*}
If $\alpha$ is a generator then either $c\neq 0$ or $d\neq 0$.

\subsection{Real Biquadratics Lower Bounds}

\begin{prop}\label{Liouvillerealbound} 
Assume that $0<l<m<n$ are pairwise coprime square-free integers. For $K=\mathbb{Q}(\sqrt{ml}, \sqrt{nl})$, we have
\[ \tfrac{1}{48} n \leq M(\mathcal{O}_K).\]
\end{prop}

\begin{proof} Suppose $\alpha=\alpha_1\in\mathcal{O}_K$ generates $K$ with the notation above.	We will consider the cases $d=0$ and $d\neq 0$ separately.  First, assume that $d\ne 0$. By choosing a suitable conjugate of $\alpha_1$, we may assume that $d>0$ and $c\ge0$. Additionally, we can assume that $\alpha_1\ge \frac{d}{4}\sqrt{mn}$. To see this, if $\alpha_1< \frac{d}{4}\sqrt{mn}$, then $a+b\sqrt{ml}+c\sqrt{nl}<0$ and
\[-\alpha_3=\tfrac{1}{4}(-a-b\sqrt{ml}+c\sqrt{nl}+d\sqrt{mn})> \tfrac{1}{4}(2c\sqrt{nl}+d\sqrt{mn})\ge \tfrac{d}{4}\sqrt{mn},\] and we can consider the primitive element $-\alpha_3$, which will satisfy the assumptions and have the same Mahler measure as $\alpha_1$.  Similarly, if $d=0$ then, $c\neq 0$, and, as above, we may assume that $c>0$ and $\alpha_1\ge \frac{c}{4}\sqrt{nl}$.

Let $\lambda$ be such that $n^{\lambda}=\operatorname{max}\{|\alpha_2|,|\alpha_4|\}$. Then, $-n^\lambda\le \alpha_i\le n^\lambda$ for $i=2,4$, and their difference implies that  
\[|c\sqrt{nl}-d\sqrt{mn}|\le 4n^\lambda.\]
When $d=0$ this implies that  $c\sqrt{nl}\le 4n^\lambda$. It follows that 
\[M(\alpha_1)\ge |\alpha_1|n^\lambda\ge \left(\frac{c}{4}\sqrt{nl}\right)^2\geq\tfrac{1}{16}nl \geq \tfrac1{48}n\]
as required. 

We will now assume $d\neq 0$. Dividing by $d\sqrt{nl}$, we get  
\[
\left|\frac{c}{d}-\sqrt{\frac{m}{l}}\right|\le\frac{4n^\lambda}{d\sqrt{nl}}=\frac{4}{d\sqrt{l} n^{\frac12-\lambda}}.
\]
We are now in a position to use Liouville's Theorem  (Theorem~\ref{thm.Liou}).  
The value $\mu=\mu\left(\sqrt\frac ml\right)$ is 
\[\mu=\frac{1}{l+2\sqrt{ml}}\geq\frac{1}{3\sqrt{ml}}.\]
If $d<\frac{1}{4}\mu\sqrt{l}n^{\frac12-\lambda}$, then the above inequality implies that
\[\left|\frac{c}{d}-\sqrt{\frac{m}{l}}\right|<\frac{\mu}{d^2},\] 
which contradicts the inequality of Theorem \ref{thm.Liou}. So, we must have $d\ge\frac{\sqrt{l}}{4}\mu n^{\frac12-\lambda}$. Since $\alpha_1\geq \tfrac d4\sqrt{mn}$, we have
\[ |\alpha_1|\ge \frac{\sqrt{ml}}{16}\mu n^{1-\lambda} = \tfrac{1}{48} n^{1-\lambda},\] 
and consequently 
\[M(\alpha_1)\ge |\alpha_1| \max\{ |\alpha_2|, |\alpha_4|\} =|\alpha_1|n^\lambda\ge  \tfrac{1}{48}n.\]

\end{proof}

For the biquadratics $K=\Q(\sqrt{m},\sqrt{n})$ with $\gcd(m,n)=1$, the following corollary improves the exponent $\frac16$ of the discriminant on the lower bound to $\frac14$. 
\begin{cor}\label{cor:realbiquadlis1} 
If $K=\Q(\sqrt{m},\sqrt{n})$, where $m, n>1$ are square-free and $\gcd(m,n)=1$, then 
\[ \tfrac{1}{96\sqrt2}D_K^\frac14\leq M(\mathcal{O}_K).\]
\end{cor}

\begin{proof}
In this case, $m^2n^2\leq D_K\leq 64m^2n^2$. Asssuming $n\geq m$, we have $n\geq\sqrt{mn}$, and, by Proposition~\ref{Liouvillerealbound},
$$M(\mathcal{O}_K)\geq \tfrac{1}{48}\sqrt{mn}=\tfrac{1}{96\sqrt2}\sqrt{8mn}\geq\tfrac{1}{96\sqrt2}D_K^\frac14.$$
\end{proof}

\subsection{Real Biquadratics Conditional Bounds}

We now prove Theorem~\ref{allexpreal}. 


\begin{proof}[proof of Theorem~\ref{allexpreal}]
	We can express $\tfrac{p}q$ as $\frac{t}{2(t+s+r)}$, where $0\leq r \leq s\leq t$ are integers. For instance, we may take $t = 2p$. If $2p \leq q - 2p$, then set $s = 2p$ and $r = q - 4p$. Otherwise, set $s = q - 2p$ and $r = 0$.
	
	First, assume $0<r<s<t$. Define the following
	\[
		l(x)=x^{2r}+2,\ \ m(x)=x^{2s}+2x^{2s-2r}+2,\ \ n(x)=x^{2t}+2x^{2t-2r}+2x^{2t-2s}+2.
	\]
	Each of $l(x), m(x),$ and $n(x)$ is Eisenstein and thus irreducible. As a result, the product $l(x)m(x)n(x)$ has no repeated roots. Moreover, $l(1)m(1)n(1)=105$ is square-free. By Granville's work stated as Theorem~\ref{thm:Granville} above, 
	 assuming the ABC conjecture, there exist infinitely many positive integers $k$ such that $l(k), m(k),$ and $n(k)$ are square-free and pairwise coprime. 
	
	For such $k$, consider the field $K = \mathbb{Q}(\sqrt{ml}, \sqrt{nl})$, where $l = l(k), m = m(k),$ and $n = n(k)$. We claim that as $k$ varies, this family of fields satisfies the stated property. The discriminant of $K$ is bounded by
	\[
	k^{4t+4s+4r} \leq D_K \leq 65k^{4t+4s+4r},
	\]
	when $k$ is large enough. In what follows, inequalities that depend on $k$ being large enough will not be explicitly stated, as the context will make it clear. Using Proposition~\ref{Liouvillerealbound}, we find
	\[
	\frac{1}{48(65^\frac{p}{q})}D_K^{\frac{p}q} \leq   \tfrac{1}{48}k^{2t} \leq M(\mathcal{O}_K).
	\]
	
	Next we derive an upper bound for $M(\mathcal{O}_K)$. 	Let
	\begin{align}\label{eq.alpha}
		a=\lfloor \sqrt{mn}\rfloor,\quad b = k^{t-r},\quad c = k^{s-r},\quad \alpha_1=a+b\sqrt{ml}+c\sqrt{nl}+\sqrt{mn},
	\end{align}
	and let $\alpha_2, \alpha_3, \alpha_4$ denote the other conjugates of $\alpha_1$, as defined earlier.
	It is clear that \[\left |\alpha_1\right |\le 5 k^{t+s}.\] 
	Observe that
	\[
	\left| b\sqrt{ml} - c\sqrt{nl} \right| = \left| \frac{l(b^2m - c^2n)}{b\sqrt{ml} + c\sqrt{nl}} \right| = \left| \frac{2lk^{2s-2r}}{b\sqrt{ml} + c\sqrt{nl}} \right| \leq 3,
	\]
	\[
	\left| c\sqrt{nl} - \sqrt{mn} \right| = \left| \frac{n(c^2l - m)}{c\sqrt{nl} + \sqrt{mn}} \right| = \left| \frac{2n}{c\sqrt{nl} + \sqrt{mn}} \right| \leq 3k^{t-s}.
	\]	
	From these, we obtain that
	\begin{align*}
	\left |\alpha_2\right |& =\left|a-b\sqrt{ml}+c\sqrt{nl}-\sqrt{mn}\right|\le \left|a-\sqrt{mn}\right|+\left|b\sqrt{ml}-c\sqrt{nl}\right|\le 4\\
	\left |\alpha_3\right |&=\left|a+b\sqrt{ml}-c\sqrt{nl}-\sqrt{mn}\right|\le \left|a-\sqrt{mn}\right|+\left|b\sqrt{ml}-c\sqrt{nl}\right|\le 4\\
	\left |\alpha_4\right |&=\left | a-\sqrt{mn}-b\sqrt{ml}+c\sqrt{nl}-2c\sqrt{nl}+2\sqrt{mn} \right |\\
		&\le \left | a-\sqrt{mn}\right|+\left| b\sqrt{ml}-c\sqrt{nl}\right|+2\left|c\sqrt{nl}-\sqrt{mn} \right |
		\le 7k^{t-s}.
	\end{align*}
	Therefore, $$M(\mathcal{O}_K)\le M(\alpha_1)\le 560k^{2t}\le 560 \: D_K^{\frac{p}q}.$$

	It remains to consider those cases when $0\leq r \leq s \leq t$ which do not satisfy $0<r<s<t$. We define the polynomials $l(x)$, $m(x)$, and $n(x)$ as follows: \\
	If $0<r<s=t$ we let 
	\[ l(x)=x^{2r}+2,\quad m(x)=x^{2t}+2x^{2t-2r}+2,\quad n(x)=x^{2t}+2x^{2t-2r}+10.\]
	If $0<r=s<t$ we let 
	\[ l(x)=x^{2r}+2,\quad m(x)=x^{2r}+6,\quad n=x^{2t}+6x^{2t-2r}+6.\]
	If $0=r<s<t$ we let
	\[ l(x)=1,\quad m(x)=x^{2s}+2,\quad n(x)=x^{2t}+2x^{2t-2s}+2.\]
	If $0<r=s=t$ we let
	\[ l(x)=x,\quad m(x)=x+1,\quad n(x)=x+2.\]
	If $0=r<s=t$ we let
	\[ l(x)=1,\quad m(x)=x,\quad n(x)=x+1.\]
	If  $0=r=s<t$ we let
	\[ l(x)=1,\quad m(x)=2, \quad n(x)=x^2+1.\]
	The families of fields $K$ considered in these cases are defined analogously using $l(x)$, $m(x)$, and $n(x)$ as in the case $0 < r < s < t$. The bounds for $D_K$ and the lower bound for $M(\mathcal{O}_K)$ are obtained similarly, requiring only minor adjustments. The choice of $\alpha_1$ used to obtain an upper bound is as follows. In the first four cases, $\alpha_1$ is taken as in \eqref{eq.alpha}. In the fifth case, we take
	\[
	\alpha_1 = \sqrt{k} + \sqrt{k+1}.
	\]
	In the last case, we take
	\[
	\alpha_1 = k\sqrt{2} + \sqrt{2(k^2 + 1)}.
	\]
	The details are similar to the general case.
\end{proof}

\subsection{Real Biquadratics Unconditional Bounds}

We obtain the following unconditional corollaries for the exponents of $\tfrac16$, $\tfrac14$, and $\tfrac12$.
\begin{cor}\label{cor:realbiqadratic16}
There are infinitely many positive integers $k$ so that $k(k+1)(k+2)$ is square-free. For large enough such k, the fields $K_k=\mathbb{Q}(\sqrt{k(k+1)},\sqrt{k(k+2)})$ satisfy
\[
4^{-\frac{2}{3}} D_K^\frac16\leq M(\mathcal{O}_K)\leq80 \: D_K^\frac16.
\]
\end{cor}
\begin{proof}
    When $0<r=s=t$ in the proof of Theorem \ref{allexpreal}, we have $l(x)m(x)n(x)=x(x+1)(x+2)$, which is of degree $3$.  By Hooley  \cite{MR214556}  there are infinitely many positive integers $k$ such that $k(k+1)(k+2)$ is square-free. The lower bound follows from $(***)$, and proceeding as in the proof of Theorem \ref{allexpreal}  yields the upper bound.
\end{proof}

\begin{cor}\label{cor:realbiqadratic14}
There are infinitely many positive integers $k$ so that $k(k+1)$ is square-free. For large enough such $k$, the fields $K_k=\mathbb{Q}(\sqrt{k},\sqrt{k+1})$ satisfy
\[
\tfrac{1}{137}D_K^\frac14\leq M(\mathcal{O}_K)\leq5D_K^\frac14.
\]
\end{cor}
\begin{proof}
    If $0=r<s=t$ in the proof of Theorem \ref{allexpreal}, then $l(x)m(x)n(x)=x(x+1)$, which is of degree 2. So, there are infinitely many positive integers $k$ such that $k(k+1)$ is square-free. The rest of the proof follows from the proof of Theorem \ref{allexpreal}.
\end{proof}

\begin{cor}\label{cor:realbiqadratic12}
For any square-free integer $k>1$ with $\gcd(2,k)=1$, the fields $K_k=\mathbb{Q}(\sqrt{2},\sqrt{k})$ satisfy
\[
\tfrac{1}{768} D_K^\frac12\leq M(\mathcal{O}_K)\leq D_K^\frac12.
\]
\end{cor}
\begin{proof}
    The discriminant satisfies $D_{K_k}^\frac{1}{2}\leq16k$. By Proposition \ref{Liouvillerealbound}, we have $\frac{1}{48}k\leq M(\mathcal{O}_K)$. The upper bound follows from $(***)$.
\end{proof}




\section{Imaginary Biquadratics}\label{section:complexbiquads}

When representing totally imaginary biquadratics, we will depart slightly from the notation in Section~\ref{section:introbiquad} and write $K=\mathbb{Q}(\sqrt{-ml}, \sqrt{-nl})$, where $m,n,$ and $ l$ are positive pairwise coprime square-free integers. We assume that $l \leq m\leq n$.
The discriminant satisfies $l^2m^2n^2\leq D_K \leq 64l^2m^2n^2$. The theoretical bounds  are given in $(*)$ and $(**)$ and the exponents on the discriminant are $\tfrac15$ and $1$. 

Define
\[
S = \Bigl\{\,
\tfrac{1}{4}(a + b\sqrt{-ml} + c\sqrt{-nl} + d\sqrt{mn})
\;:\; a, b, c, d \in \mathbb{Z} \Bigr\}.
\]
From Section~\ref{section:introcyclic}, we know that $\mathcal{O}_K \subseteq S$. Consequently, 
$$\min_{\alpha\in S}\{M'(\alpha) \;:\; \mathbb{Q}(\alpha)=K\} \le M(\mathcal{O}_K),$$
where $M'$ is given as in the equation \eqref{eq.M'}.
Our next goal is to obtain a lower bound for $M(\mathcal{O}_K)$ by deducing a lower bound for the right-hand side of the last inequality.
Let $\alpha = \alpha_1 \in S$. Then there exist $a, b, c, d \in \mathbb{Z}$ such that
\[
\alpha_1 = \tfrac{1}{4}(a + b\sqrt{-ml} + c\sqrt{-nl} + d\sqrt{mn}).
\]
The other conjugates of $\alpha_1$ are 
\begin{align*}
\alpha_2 & = \tfrac{1}{4}(a-b\sqrt{-ml}-c\sqrt{-nl} + d\sqrt{mn}) \\
 \alpha_3 & = \tfrac{1}{4}(a-b\sqrt{-ml}+c\sqrt{-nl} - d\sqrt{mn} )\\
 \alpha_4 & = \tfrac{1}{4}(a+b\sqrt{-ml}-c\sqrt{-nl} - d\sqrt{mn}).
\end{align*}
Using this notation, $\alpha_2=\overline{\alpha}_1$ and $\alpha_4 = \overline{\alpha}_3$.
Observe that 
\[ |\alpha_1|^2=\alpha_1\alpha_2 = \tfrac{1}{16}\big((a+d\sqrt{mn})^2+ (b\sqrt{ml}+c\sqrt{nl})^2\big) \]
and 
\[ |\alpha_3|^2=\alpha_3\alpha_4 = \tfrac{1}{16}\big((a-d\sqrt{mn})^2+ (b\sqrt{ml}-c\sqrt{nl})^2\big).\]

We encode $\alpha$ by the four-tuple of coefficients $(a,b,c,d)$.
It then suffices to minimize $M'(\alpha)$ over the three families
\begin{equation}\label{em.3cases.biq}
	(0,b,c,0), \quad (a,1,0,d), \quad (a,0,1,d),
\end{equation}
where $b,c,d\neq 0$. 
The nonzero condition ensures that $\alpha$ is a primitive element in the biquadratic field. The reason of using $M'$ rather than $M$ is that elements of these families may not be integral. For example, consider the field $K = \mathbb{Q}(\sqrt{-7}, \sqrt{-14})$.  
Take the element
\[
\alpha = \tfrac{1}{2}(1 + \sqrt{-7} + \sqrt{-14} + \sqrt{2}) \in \mathcal{O}_K,
\]
for which $M'(\alpha) \approx 11.66$.  
On the other hand,
\[
\beta = \tfrac{1}{2}(\sqrt{-7} + \sqrt{-14}) \in S
\]
gives $M'(\beta) \approx 10.20$.  
Although $M'(\beta) < M'(\alpha)$, note that $\beta \notin \mathcal{O}_K$.

\subsection{Lower Bounds}
The following theorem provides a general lower bound for $M(\mathcal{O}_K)$, with the imaginary biquadratic $K$ as above, in terms of $l$ and $n$. 

\begin{prop}\label{Liouvilleimaginarybound} Let $K=\mathbb{Q}(\sqrt{-ml}, \sqrt{-nl})$ be a biquadratic field, where $m,n,l>0$ are pairwise coprime square-free integers with $n > m$. Then
\[
  M(\mathcal{O}_K)\geq
  \begin{cases}
     \frac{1}{256}ln &\text{if $n>l$,}\\
     \frac{1}{2304}l^2 & \text{if $l> n$}. 
  \end{cases}
  \]
 In general, $\frac{1}{256}ln \leq M(\mathcal{O}_K).$ 
\end{prop}
\begin{proof}

If $\alpha=\alpha_1$ is represented by the four-tuple $(a,1,0,d)$, we can assume $a,d>0$ by choosing a preferred conjugate. Then, $|\alpha_3|^2\geq\frac{1}{16}lm$ and
\[
  |\alpha_1|^2\geq
  \begin{cases}
     \tfrac{1}{16}mn &\text{if $n> l$,}\\
     \tfrac{1}{16}lm & \text{if $l> n$}, 
  \end{cases}
  \]
and hence 
\[
  M'(\alpha)\ge|\alpha_1|^2|\alpha_3|^2\geq
  \begin{cases}
     \tfrac{1}{256}lm^2n &\text{if $n> l$,}\\
     \tfrac{1}{256}l^2m^2 & \text{if $l> n$}. 
  \end{cases}
  \]
If $\alpha_1$ is represented by $(a,0,1,d)$, we assume $a,d>0$. Then, $$M'(\alpha)\ge|\alpha_1|^2|\alpha_3|^2\geq\left(\tfrac{1}{16}ln\right)^2=\tfrac{1}{256}l^2n^2.$$
Now, we consider those $\alpha_1$ represented by $(0,b,c,0)$ and assume $b,c>0$. Then, 
$$M'(\alpha)\geq|\alpha_1|^2\geq\tfrac{1}{16}ln.$$
For $l> n$, we again use Liouville's theorem. Clearly, $|\alpha_1|^2\geq\frac{1}{16} c^2ln$. Let $\lambda$ be a number such that $l^\lambda=|\alpha_3|=\frac{1}{4}|b\sqrt{ml}-c\sqrt{nl}|$. Then, $$\left|\sqrt{\frac nm}-\frac{b}{c}\right|=\frac{4l^\lambda}{c\sqrt{ml}}=\frac{4}{cm^{1/2}l^{1/2-\lambda}}.$$
Liouville's theorem implies that $c\geq \frac{1}{4}\mu m^\frac12 l^{\frac12-\lambda}$, where 
$$\mu=\frac{1}{m+2\sqrt{mn}}\geq\frac{1}{3\sqrt{mn}}.$$
Therefore,  $c\geq \frac{1}{12}n^{-\frac12} l^{\frac12-\lambda}$, and hence $$M'(\alpha)\geq |\alpha_1|^2|\alpha_3|^2\geq \tfrac{1}{16} \left(\tfrac{1}{12}n^{-\frac12} l^{\frac12-\lambda}\right)^2nl^{1+2\lambda}=\tfrac{1}{2304}l^2.$$
\end{proof}

We can now improve the exponent $\frac15$ on the discriminant on the left-hand side of $(*)$ and $(**)$ to $\frac14$ with a $\sqrt{l}$ factor. 

\begin{cor}\label{real>14withl}
Let $K=\mathbb{Q}(\sqrt{-ml}, \sqrt{-nl})$ be a biquadratic field, where $m,n,l>0$ are pairwise coprime square-free integers. Then
\[  \frac{\sqrt{l}}{512\sqrt{2}} D_K^\frac14 \leq M(\mathcal{O}_K).\]
\end{cor}

\begin{proof} We may assume $n>m$. Then, by Proposition~\ref{Liouvilleimaginarybound},
$$M(\mathcal{O}_K)\geq\tfrac{1}{256}ln\geq\tfrac{\sqrt{l}}{512\sqrt2}\sqrt{8lmn}\geq\tfrac{\sqrt{l}}{512\sqrt2}D_K^\frac14.$$\end{proof}

An immediate consequence of Corollary~\ref{real>14withl} is that we can improve the exponent $\frac15$ on the discriminant on the left-hand side of $(*)$ and $(**)$ to $\frac14$. Specifically, we have the following.
\begin{theorem}\label{real>14}
Let $K=\mathbb{Q}(\sqrt{-ml}, \sqrt{-nl})$ be a biquadratic field, where $m,n,l>0$ are pairwise coprime square-free integers. Then
\[  \tfrac{1}{512\sqrt{2}} \: D_K^{\frac14} \leq M(\mathcal{O}_K).\]
\end{theorem}

The exponent can be further improved for the following family.

\begin{cor}\label{cor-mn} Let $K=\mathbb{Q}(\sqrt{-m}, \sqrt{n})=\mathbb{Q}(\sqrt{-m},\sqrt{-mn})$ be a biquadratic field, where $m,n>0$ are square-free and $gcd(m,n)=1$. Then 
\begin{equation*}
\tfrac{1}{2048} D_K^\frac12 \leq M(\mathcal{O}_K).\end{equation*}
\end{cor}
\begin{proof}
By Proposition~\ref{Liouvilleimaginarybound},
\[
\tfrac{1}{2048} D_K^\frac12 \leq \tfrac{1}{2048}(8mn) = \tfrac{1}{256}mn \leq M(\alpha).
\]
\end{proof}

\subsection{Conditional Bounds}

Assuming ABC conjecture, we now show that for any rational exponent $\frac{p}{q}$ between $\tfrac14$ and 1 that there are infinitely many biquadratics whose minimal Mahler measure is a constant times $D_K^{\frac{p}{q}}$, proving Theorem~\ref{thm:conditionalimaginary}.


\begin{proof}[proof of Theorem~\ref{thm:conditionalimaginary}]
We prove this in two cases, the first for $\frac14\leq\frac pq\leq\frac12$ and the second for $\tfrac12 \leq \tfrac{p}{q} \leq 1$.

We first consider the case when  $\frac14\leq\frac pq\leq\frac12$. In this case $\tfrac{p}{q}$ can be written as $\frac pq=\frac{t}{2(t+s)}$ with $0\leq s \leq t$. One can choose $t=2p$ and $s=q-2p$.

Assume first $s<t$. By Eisenstein's criterion, $m(x):=x^{2s}+2$ and $n(x):=x^{2t}+2x^{2t-2s}+2$ are both irreducible, and observe that they do not share roots; if $x^{2s}+2=0$, then $x^{2t}+2x^{2t-2s}+2=2$ . Also, $f(x)=m(x)n(x)$ has $f(1)=15$, which is square-free. 
By Granville's work stated as Theorem~\ref{thm:Granville}, assuming the ABC conjecture, 
there are infinitely many positive integers $k$ with $f(k)$ square-free. For these $k$, let $m=m(k)$ and $n=n(k)$ and 
\[K=\Q(\sqrt{-m},\sqrt{-n})=\mathbb{Q}\left(\sqrt{-(k^{2s}+2)}, \sqrt{-(k^{2t}+2k^{2t-2s}+2)}\right).\]
Then, $k^{4t+4s}\leq D_K\leq 65 k^{4t+4s}$ for sufficiently large $k$. All inequalities involving $k$ below are assumed to hold 
for sufficiently large $k$, as needed. By Proposition~\ref{Liouvilleimaginarybound}, 
\[ M(\mathcal{O}_K)\geq\tfrac{1}{256}(k^{2t}+2k^{2t-2s}+2)\geq\tfrac{1}{256}k^{2t}\geq\frac{1}{256(65^\frac{p}{q})}D_K^{\frac pq}.\]

To get an upper bound, consider the algebraic integer $\alpha_1$ represented by $(0,4k^{t-s},4,0)$. Then, 
\[|\alpha_1|^2=\tfrac{1}{16}\left(4k^{t-s}\sqrt{k^{2s}+2}+4\sqrt{k^{2t}+2k^{2t-2s}+2}\right)^2\leq5k^{2t},\]
and 
\begin{align*}
 |\alpha_3|^2
 & =\left(k^{t-s}\sqrt{k^{2s}+2}-\sqrt{k^{2t}+2k^{2t-2s}+2}\right)^2 \\
 & =\frac{4}{\left(k^{t-s}\sqrt{k^{2s}+2}+\sqrt{k^{2t}+2k^{2t-2s}+2}\right)^2} < 1.
 \end{align*}
Thus,
\begin{equation}\label{bi-im-con-up}
M(\mathcal{O}_K)\leq M(\alpha_1)\leq 5k^{2t}\leq 5 D_K^{\frac pq}.
\end{equation}

For $t=s$, consider $m(x)=x$ and $n(x)=x+1$, and consider the algebraic integer $\alpha_1$ represented by $(0,4,4,0)$.

Now we consider the case where $\frac12\leq\frac pq\leq1$.
We write $\frac pq$ as $\frac pq=\frac{t}{t+s}$ with $0\leq s \leq t$. Note that $m(x)=x^s+1$ and $n(x)=x^t+2$ are both irreducible and do not share roots. For $f(x)=m(x)n(x)$, we have $f(1)=6$, which is square-free. Again,  by Theorem~\ref{thm:Granville}, 
there are infinitely many positive integers $k$ with $f(k)$ square-free. For these $k$, let $m=m(k)$ and $n=n(k)$ and 
\[
K=\Q(\sqrt{-n},\sqrt{m}) = \Q(\sqrt{-n},\sqrt{-mn}).
\]
Again, for large enough $k$, $k^{2s+2t}\le D_K\le 65k^{2s+2t}$. By Theorem \ref{Liouvilleimaginarybound},
\[M(\mathcal{O}_K)\geq\tfrac{1}{2304}(k^t+2)^2\geq\tfrac{1}{2304}k^{2t}\geq\frac{1}{2304(65^\frac pq)}D_K^{\frac pq}.\]

For an upper bound, we take the algebraic integer $\alpha_1$ represented by $(0,4,0,4)$. Then,
$$|\alpha_1|^2, |\alpha_3|^2=k^t+k^s+3\leq 3k^t,$$ and hence 
\begin{equation*}
M(\mathcal{O}_K)\leq9k^{2t}\leq9D_K^\frac pq.\end{equation*}

\end{proof}

\subsection{Unconditional Bounds}

We now show unconditionally that we have infinitely many number fields achieving the exponents of  $ \tfrac14, \tfrac12$, $\tfrac23$ and $1$.

\begin{cor}\label{cor:imagbiqadratic14}
    There are infinitely many positive integers $k$ such that $k(k+1)$ is square-free. For such $k>2$, the fields $K_k=\mathbb{Q}(\sqrt{-k},\sqrt{-(k+1)})$ satisfy
    \begin{equation*}
       \tfrac{1}{512\sqrt{2}}D_{K_k}^\frac{1}{4}\leq M(\mathcal{O}_{K_k})\leq 5D_{K_k}^\frac{1}{4}.
    \end{equation*}
\end{cor}
\begin{proof}
Considering the case $\frac{1}{4}\leq\frac{p}{q}\leq\frac{1}{2}$ with $s=t$ in the proof of Theorem \ref{thm:conditionalimaginary}, we have $f(x)=x(x+1)$, which is of degree $2$. So, there are infinitely many positive integers $k$ such that $k(k+1)$ is square-free. Theorem \ref{real>14}  and the  inequality \eqref{bi-im-con-up} complete the proof.
\end{proof}

\begin{cor}\label{cor:imagbiqadratic12}
    There are infinitely many positive integers $k$ so that $k(k+1)$ is square-free. For such $k>1$, the fields $K_k=\mathbb{Q}(\sqrt{-(k+1)},\sqrt{-k(k+1)})$ satisfy
    \begin{equation*}
       \tfrac{1}{2048}D_{K_k}^\frac{1}{2}\leq M(\mathcal{O}_{K_k})\leq 9\: D_{K_k}^\frac{1}{2}.
    \end{equation*}
\end{cor}
\begin{proof}
As in the previous corollary, there are infinitely many positive integers $k$ such that $k(k+1)$ is square-free. The lower bound follows from Corollary \ref{cor-mn}, and proceeding as in the case $\frac12\leq\frac pq\leq1$ with $s=t=1$ in the proof of Theorem \ref{thm:conditionalimaginary} yields the upper bound.
\end{proof}

\begin{cor}\label{cor:imagbiqadratic23}
    There are infinitely many positive integers $k$ so that $k(k^2+1)$ is square-free. For large enough such $k>1$, the fields $K_k=\mathbb{Q}(\sqrt{-(k^2+1)},\sqrt{-k(k^2+1)})$ satisfy 
    $$\tfrac{1}{37248}D_{K_k}^\frac{2}{3}\leq M(\mathcal{O}_{K_k})\leq4 D_{K_k}^\frac{2}{3}.$$
\end{cor}
\begin{proof}
Since $f(x)=x(x^2+1)$ is of degree $3$, by Hooley's work \cite{MR214556}, there are infinitely many positive integers $k$ such that $k(k^2+1)$ is square-free. Proceeding as in the case $\frac12\leq\frac pq\leq1$ with $s=1,\ t=2$ in the proof of Theorem \ref{thm:conditionalimaginary} completes the proof.
 \end{proof}
\begin{cor}\label{cor:imagbiqadratic1}
    For any square-free integer $k>1$ with $\gcd(k,2)=1$, the fields $K_k=\mathbb{Q}(\sqrt{-2k},\sqrt{-k})$ satisfy $$\tfrac{1}{589824}D_{K_k}\leq M(\mathcal{O}_{K_k})\leq D_{K_k}.$$ 
\end{cor}
\begin{proof}
    The discriminant satisfies $D_{K_k}\leq256k^2$. By Proposition \ref{Liouvilleimaginarybound}, we have $\frac{1}{2304}k^2\leq M(\mathcal{O}_{K_k})$. The upper bound follows from $(*)$ and $(**)$.
\end{proof}


 

\subsection{Biquadratic fields containing non-trivial roots of unity}\label{section:rootsunity}

Due to work of  Akhtari, Vaaler, and Widmer \cite[Corollary 1.1]{AVW1}, there are different theoretical upper bounds for $M(\Ox_K)$ amongst totally imaginary biquadratic $K$ which contain either  $\sqrt{-1}$ or $\sqrt{-3}$ and those that do not. Specifically, they show that with $c_K=(\tfrac{2}{\pi})^{r_2}D_K^{\frac12}$  if $\sqrt{-1}, \sqrt{-3}\not\in K$ then $M(\Ox_K)\leq c_K$. Further, if $c_K\leq M(\Ox_K)$ then neither $\sqrt{-1}$ nor $\sqrt{-3}$ is in $K$.   

In this section we determine $\alpha_1\in \Ox_K$ such that $M(\alpha_1) \leq c_K$ for  biquadratic number fields which contain either $\sqrt{-1}$ or $\sqrt{-3}$. The field $\Q(\sqrt{-1}, \sqrt{-3})=\Q(\zeta_{12})$ is the only such field containing both. The other fields can be written as $\Q(\sqrt{-1}, \sqrt{-k})$, or $\Q(\sqrt{-3}, \sqrt{-k})$ where $k$ is square-free, where we allow for the case when $3$ divides $k$.

Consider $K =  \Q(\sqrt{-1}, \sqrt{-k})$. 
If $k\equiv 1\pmod 4$, let
\[ \alpha_1=\tfrac{1}{2}\left((\lfloor{\sqrt{k}}\rfloor+\epsilon)\sqrt{-1}+\sqrt{-k}\right),\] where $\epsilon\in\{0,1\}$ is determined uniquely by the condition that $\lfloor{\sqrt{k}}\rfloor+\epsilon$ is odd. 
Then $M(\alpha_1) \leq c_K$ for $k\geq 4$. 
If $k\equiv2\pmod4$, let
\[ \alpha_1=\tfrac{1}{2}\left(\lfloor{\sqrt{k}}\rfloor+\epsilon+(\lfloor{\sqrt{k}}\rfloor+\epsilon)\sqrt{-1}+\sqrt{-k}+\sqrt{k}\right),\] where $\epsilon\in\{0,1\}$ is determined uniquely by the condition that $\lfloor{\sqrt{k}}\rfloor+\epsilon$ is even. Then, $M(\alpha_1) \leq c_K$ for all $k\geq 4$. If $k\equiv3\pmod{4}$, then $(-1,k)\equiv(3,3)\pmod{4}$ so that we work with $K=\mathbb{Q}(\sqrt{-1},\sqrt{k})$ instead of $\mathbb{Q}(\sqrt{-1},\sqrt{-k})$. Let
\[ \alpha_1=\tfrac{1}{2}\left(1+(\lfloor{\sqrt{k}}\rfloor+\epsilon)\sqrt{-1}+\sqrt{-k}\right),\] where $\epsilon\in\{0,1\}$ is determined uniquely by the condition that $\lfloor{\sqrt{k}}\rfloor+\epsilon$ is odd. Then $M(\alpha_1) \leq c_K$. 
When $k = 2$, we take 
$\alpha_1 = \tfrac{1}{2}(\sqrt{-2} + 2\sqrt{-1} + \sqrt{2})$.
When $k = 3$, we take 
$\alpha_1 = \tfrac{1}{2}(1 + \sqrt{-1} + \sqrt{3} + \sqrt{-3})$.
 In both cases, one can verify that $M(\alpha_1)\leq c_K$.

Consider $K =  \Q(\sqrt{-3}, \sqrt{-k})$ with $\gcd(3,k)=1$. 
If $k\equiv3\pmod{4}$, then let
\[ \alpha_1=\tfrac{1}{4}\Big((2\Big\lfloor{\sqrt{\tfrac{k}{3}}} \Big\rfloor+\epsilon)\sqrt{-3}+2\sqrt{-k}\Big),\] 
where $\epsilon\in\{0,2\}$ is determined uniquely by the condition that $2\Big\lfloor{\sqrt{\tfrac{k}{3}}} \Big\rfloor+\epsilon\equiv2\pmod{4}$. Then, $M(\alpha_1) \leq c_K$  for all $k\geq72.$
If  $\epsilon=0$ then $M(\alpha_1)\leq c_K$. Those $k\le 71$ with $\epsilon=2$ are $k=19,23, 55, 59, 67, 71.$  If $k\equiv2\pmod{4}$, then with 
\[ \alpha_1=\tfrac{1}{2}\Big(2\Big\lfloor{\sqrt{\tfrac{k}{3}}} \Big\rfloor\sqrt{-3}+2\sqrt{-k}\Big)\]
we have $M(\alpha_1) \leq c_K$ for all $k\geq1$.
If $k\equiv1\pmod{4}$, then $(3k,-k)\equiv(3,3)\pmod{4}$ so  we work with $K=\mathbb{Q}(\sqrt{3k},\sqrt{-k})$ instead of $\mathbb{Q}(\sqrt{-3},\sqrt{-k})$. With 
\[ \alpha_1=\tfrac{1}{2}\Big(2\sqrt{-k}+2\Big\lfloor{\sqrt{\tfrac{k}{3}}} \Big\rfloor\sqrt{-3}\Big),\]  then $M(\alpha_1) \leq c_K$ for all $k\geq1$
With $\alpha_1\in\mathcal{O}_K$ chosen as in the table below, we obtain $M(\alpha_1)\leq c_K$ in each remaining case.
 \begin{center}
 	{\renewcommand{\arraystretch}{1.6}
 		\begin{tabular}{|c|c|c|c|}
 			\hline
 			$k$ & $\alpha_1$ & $M(\alpha_1)$ & $c_K$ \\ \hline
 			\rule{0pt}{3.2ex}$19$ & $\frac{1}{4}\left(2 + 4\sqrt{-3} + 2\sqrt{-19}\right)$ & $15.55$ & $23.10$ \\ \hline
 			\rule{0pt}{3.2ex}$23$ & $\frac{1}{4}\left(2 + 4\sqrt{-3} + 2\sqrt{-23}\right)$ & $17.31$ & $27.96$ \\ \hline
 			\rule{0pt}{3.2ex}$55$ & $\frac{1}{4}\left(6\sqrt{-3} + 2\sqrt{-55}\right)$       & $49.00$ & $66.87$ \\ \hline
 			\rule{0pt}{3.2ex}$59$ & $\frac{1}{4}\left(2 + 8\sqrt{-3} + 2\sqrt{-59}\right)$ & $53.61$ & $71.74$ \\ \hline
 			\rule{0pt}{3.2ex}$67$ & $\frac{1}{4}\left(2 + 8\sqrt{-3} + 2\sqrt{-67}\right)$ & $57.35$ & $81.46$ \\ \hline
 			\rule{0pt}{3.2ex}$71$ & $\frac{1}{4}\left(2 + 8\sqrt{-3} + 2\sqrt{-71}\right)$ & $59.19$ & $86.33$ \\ \hline
 		\end{tabular}
 	}
 \end{center}

Now consider $\mathbb{Q}(\sqrt{-3},\sqrt{-k})$ with $\gcd(3,k)=3$. If $k\equiv3\pmod{4}$, then let  $$\alpha_1=\tfrac{1}{4}\Big(2\Big\lfloor{\sqrt{\tfrac{k}{3}}} \Big\rfloor+\epsilon+2\sqrt{-3}+2\sqrt{\tfrac{k}{3}}\Big),$$ where $\epsilon\in\{0,2\}$ is determined uniquely by the condition that $2\Big\lfloor{\sqrt{\tfrac{k}{3}}} \Big\rfloor+\epsilon\equiv0\pmod{4}$. Then, $M(\alpha_1) \leq c_K$  for all $k\geq 91.$
This excludes  $k=15,39,51,87$. If $k\equiv2\pmod{4}$, then we choose $$\alpha_1=\tfrac{1}{2}\Big(2\Big\lfloor{\sqrt{\tfrac{k}{3}}} \Big\rfloor+1+\sqrt{-3}+2\sqrt{\tfrac{k}{3}}\Big).$$
 Then, $M(\alpha_1) \leq c_K$ for all $k\geq23$. This bound excludes $k=6$. If $k\equiv1\pmod{4}$, $(k/3,-k)\equiv(3,3)\pmod{4}$ so this time we work with $K=\mathbb{Q}(\sqrt{k/3},\sqrt{-k})$ instead of $\mathbb{Q}(\sqrt{-3},\sqrt{-k})$.  We choose $$\alpha_1=\tfrac{1}{2}\Big(2\Big\lfloor{\sqrt{\tfrac{k}{3}}} \Big\rfloor+1+2\sqrt{\tfrac{k}{3}}+\sqrt{-3}\Big).$$
Then, $M(\alpha_1) \leq c_K$ for all $k\geq23$.  The only missing integer in this case is $k= 21$. With $\alpha_1\in\mathcal{O}_K$ chosen as in the table below, we obtain $M(\alpha_1)\leq c_K$ in each remaining case.
\begin{center}
	{\renewcommand{\arraystretch}{1.6}
		\begin{tabular}{|c|c|c|c|}
			\hline
			$k$ & $\alpha_1$ & $M(\alpha_1)$ & $c_K$ \\ \hline
			\rule{0pt}{3.2ex}$6$ & $\frac{1}{2}\left( \sqrt{2}+\sqrt{-6}  \right)$ & $4.00$ & $9.73$ \\ \hline
			\rule{0pt}{3.2ex}$15$ & $\frac{1}{4}\left(2\sqrt{-3} + 2\sqrt{5}\right)$       & $4.00$ & $6.08$ \\ \hline
			\rule{0pt}{3.2ex}$21$ & $\frac{1}{2}\left(2 + 2\sqrt{-3} + \sqrt{7}+\sqrt{-21}\right)$ & $21.58$ & $34.04$ \\ \hline
			\rule{0pt}{3.2ex}$39$ & $\frac{1}{4}\left(4 + 2\sqrt{-3} + 2\sqrt{13}\right)$ & $12.00$ & $15.81$ \\ \hline
			\rule{0pt}{3.2ex}$51$ & $\frac{1}{4}\left(8 + 2\sqrt{-3} + 2\sqrt{17}\right)$ & $17.25$ & $20.67$ \\ \hline
			\rule{0pt}{3.2ex}$87$ & $\frac{1}{4}\left(8 + 2\sqrt{-3} + 2\sqrt{29}\right)$ & $28.00$ & $35.26$ \\ \hline
		\end{tabular}
	}
\end{center}

The only cyclic quartic fields with roots of unity other that $\pm 1$ are splitting fields of the cyclotomic polynomials $\Phi_5= x^4+x^2+x+1$ and $\Phi_{10}=x^4-x^3+x^2-x+1$. In both of these cases, the discriminant of the field is $5^3$ and so the $c_K$ value is $c_K=(\tfrac{2}{\pi})^{2}5^{\frac32}=4.531\dots$ and the minimal Mahler measure of a generator is 1 as they are cyclotomics, and $1 \leq c_K$.



\section{Real Cyclic Quartics}\label{section:rcyclic}

As in Section~\ref{section:introcyclic} we  write a real cyclic quartic number field as $K=(\sqrt{A(D+B\sqrt{D})})$ where $A,B,C,$ and $D$ are rational integers which satisfy $A>0$ and $D=B^2+C^2$ with $B>0$. 
The theoretical bounds for real cyclic quartics are 
\[
4^{-\frac23}D_K^{\frac16} \leq M(\Ox_K) \leq D_K^{\frac12}.
\]
From Section~\ref{section:introcyclic}, an element $\alpha=\alpha_1\in\Ox_K$ can be written as
\[
\alpha_1=\tfrac{1}{4}(x_1+x_2\sqrt{D}+x_3\rho+x_4\sigma),
\quad x_i\in\mathbb{Z},
\]
where
\[
\rho=\sqrt{A(D+B\sqrt{D})},\qquad 
\sigma=\sqrt{A(D-B\sqrt{D})}.
\]
The other conjugates of $\alpha_1$ are
\[
\begin{aligned}
	\alpha_2&=\tfrac{1}{4}(x_1-x_2\sqrt{D}-x_4\rho+x_3\sigma),\\
	\alpha_3&=\tfrac{1}{4}(x_1+x_2\sqrt{D}-x_3\rho-x_4\sigma),\\
	\alpha_4&=\tfrac{1}{4}(x_1-x_2\sqrt{D}+x_4\rho-x_3\sigma).
\end{aligned}
\]
First, we establish a lower bound for the integral Mahler measure which depends on $A$ and $D$.

\begin{prop}\label{Liouvillerealbound_cyc}
	Let $K=\mathbb{Q}(\sqrt{A(D+B\sqrt{D})})$ be a real cyclic quartic field, where $A,B,C,$ and $D$ with $A>0$ satisfy the conditions Section~\ref{section:introcyclic}. Then
	\[ \tfrac1{48} A\sqrt{D} \leq 
	M\bigl(\mathcal{O}_K\bigr).
	\]
\end{prop}
\begin{proof}
	Let $\rho$, $\sigma$ be as in \eqref{rho_sig}, and assume that $\alpha_1 \in \mathcal{O}_K$ generates $K$. Then there exist $x_1, x_2, x_3, x_4 \in \mathbb{Z}$ such that
	\[
	\alpha_1 = \tfrac{1}{4}\left(x_1 + x_2\sqrt{D} + x_3\rho + x_4\sigma\right),
	\]
	where at least one of $x_3$, $x_4$ is nonzero. Let $\alpha_2$, $\alpha_3$, $\alpha_4$ be the other conjugates of $\alpha_1$ as defined earlier. By selecting an appropriate conjugate of $\alpha_1$ or $-\alpha_1$, we may assume, without loss of generality, that $x_3 > 0$ and $x_4 \ge 0$. Moreover, we can impose the condition
	\begin{equation}\label{eq.alpbd}
		\alpha_1 \ge \tfrac{1}{4}x_3\rho,
	\end{equation}
	for the following reason. Suppose instead that $\alpha_1 < \frac{1}{4}x_3\rho$. Then
	\[
	\tfrac{1}{4}(-x_1 - x_2\sqrt{D} - x_4\sigma) > 0,
	\]
	from which it follows that
	\[
	-\alpha_3 = \tfrac{1}{4}(-x_1 - x_2\sqrt{D} + x_3\rho + x_4\sigma) > \tfrac{1}{4}(x_3\rho + 2x_4\sigma) \ge \tfrac{1}{4}x_3\rho.
	\]
	We may then work with $-\alpha_3$, which satisfies our assumptions and has the same Mahler measure as $\alpha_1$.
	
	Set $\lambda = \max\{|\alpha_2|, |\alpha_4|\}$. Then for $i = 2, 4$ we have $-\lambda \le \alpha_i \le \lambda$, so taking the difference yields
	\[
	|x_4\rho - x_3\sigma| \le 4\lambda.
	\]
	Dividing both sides by $x_3\rho$, we obtain
	\[
	\left| \frac{x_4}{x_3} - \frac{\sqrt{D} - B}{C} \right| \le \frac{4\lambda}{x_3\sqrt{A(D + B\sqrt{D})}}.
	\]
	
	We use Liouville's Theorem, Theorem~\ref{thm.Liou}, with  $\mu = \mu\left( \frac{\sqrt{D} - B}{C} \right)$. Because $\frac{\sqrt{D} - B}{C}$ is a root of the equation $Cx^2 + 2Bx - C = 0$, we derive
	\[
	\mu \ge \frac{1}{C + 2\sqrt{D}}.
	\]
	Theorem~\ref{thm.Liou} then gives
	\[
	x_3 \ge \frac{\mu}{4\lambda} \sqrt{A(D + B\sqrt{D})}.
	\]
	When combined with \eqref{eq.alpbd}, this establishes
	\[
	|\alpha_1| \ge \frac{\mu}{16\lambda} A(D + B\sqrt{D}) \ge \frac{A\sqrt{D}(B + \sqrt{D})}{16\lambda(C + 2\sqrt{D})}.
	\]
	Hence,
	\[
	M(\alpha_1) \ge |\alpha_1|\lambda \ge \frac{A\sqrt{D}(B + \sqrt{D})}{16(C + 2\sqrt{D})} \ge \tfrac{1}{48}A\sqrt{D},
	\]
	where the last inequality follows by taking $B = 0$ and $C = \sqrt{D}$. Since $\alpha_1$ was arbitrary, the desired bound follows.
\end{proof}

We now show that the exponent of $\tfrac16$ on the lower bound is sharp.
\begin{theorem}\label{theorem:realcyclic16}
	There are infinitely many integers $k>0$ so that $k^2+1$ is square-free. For large enough such $k$, the fields $K_k=\Q\big(\sqrt{k^2+1+\sqrt{k^2+1}} \big)$ satisfy
	\[
	4^{-\frac{2}{3}} D_K^{\frac16}\leq M(\Ox_K) \leq 60 D_K^{\frac16}.
	\]
	
\end{theorem}
\begin{proof} 
	Let 
	\[
	D(x) = x^2 + 1.
	\]
	There exist infinitely many positive integers $k$ for which $D(k)$ is square-free (see \cite{MR1512732}). For any such $k$, we set $D := D(k)$ and consider the field
	\[
	K := \mathbb{Q} \left( \sqrt{D + \sqrt{D}} \right).
	\]
	From \eqref{eq.disc_cyc}, the discriminant $D_K$ satisfies
	\[
	k^{6} \leq D_K \leq 257 k^{6}.
	\]
	The lower bound follows from $(***)$. To obtain an upper bound, consider
	\[
	\alpha_1 = \lfloor \rho \rfloor + \sqrt{D} + \rho + \sigma,
	\]
	where $\rho = \sqrt{D + \sqrt{D}}$ and $\sigma = \sqrt{D - \sqrt{D}}$. Let $\alpha_2, \alpha_3, \alpha_4$ be the other conjugates of $\alpha_1$ as given earlier. It is clear that
	\[
	\alpha_1 \leq 5k.
	\]
	Direct computations give that
	\[ |\sqrt{D} - \sigma| \leq 1,\quad |\sqrt{D} - \rho| \leq 1,\quad |\rho - \sigma| \leq 1. \]
	From these, we deduce
	\[
	|\alpha_2| \leq |\lfloor \rho \rfloor - \rho| + |\sqrt{D} - \sigma| \leq 2,
	\]
	\[
	|\alpha_3| \leq |\lfloor \rho \rfloor - \rho| + |\sqrt{D} - \sigma| \leq 2,
	\]
	\[
	|\alpha_4| \leq |\lfloor \rho \rfloor - \rho| + |\rho - \sqrt{D}| + |\rho - \sigma| \leq 3.
	\]
	Therefore
	\[
	M(\mathcal{O}_K) \leq M(\alpha_1) \leq 60k \leq 60 \:D_K^{\frac16}.
	\]
\end{proof}

We also show that the exponent of $\tfrac12$ in the upper bound is sharp.
\begin{theorem}\label{theorem:realcyclic12}
	There are infinitely many square-free integers $k$. For such an $k>0$ the fields $K_k=\Q\big(\sqrt{k(5 +\sqrt{5}} \big)$ satisfy
	\[
	\tfrac1{1920} \: D_K^{\frac12}\leq M(\Ox_K) \leq   D_K^{\frac12}.
	\]
	
\end{theorem}

\begin{proof}
	The discriminant satisfies $D_{K_k}^{\frac12}=40\sqrt{5}k$. By Proposition~\ref{Liouvillerealbound_cyc},  $\frac{\sqrt{5}}{48}k\leq M(\Ox_K)$. By $(***)$ we have that $M(\Ox_K) \leq  D_{K_k}^{\frac12}$.

\end{proof}

The theoretical bounds for these real cyclic quartics have exponents in the range $[\tfrac16, \tfrac12]$, and in Theorem~\ref{theorem:realcyclic16} we have shown that the lower bound exponent of $\tfrac16$ is sharp.

We now prove Theorem~\ref{realcyc_main} in two parts. First in Proposition~\ref{realcyc_31012}, assuming the ABC conjecture,  we will obtain all rational exponents in the range $[\tfrac3{10},\tfrac12)$. In Proposition~\ref{realcyc_14310}, assuming the ABC conjecture,  we will obtain all rational exponents in the range $(\tfrac14,\tfrac3{10})$.

\begin{prop}\label{realcyc_31012}
	Let $\tfrac{3}{10} \le \tfrac{p}{q} < \tfrac12$ be a rational number. There are absolute  constants $c_1, c_2>0$ such that assuming the ABC conjecture there are infinitely many real cyclic quartic fields $K$ 
	\[ c_1 \: D_K^{\frac{p}{q}} \le M(\mathcal{O}_K) \le c_2 \: D_K^{\frac{p}{q}}.\]
\end{prop}
\begin{proof}
	Represent $\frac{p}{q}$ in the form $\frac{2s+t}{4s+6t}$, where $s,t \in \mathbb{N}_{\geq 1}$. For example, we can take $s = 6p-q$ and $t = 2q-4p$. Define the polynomials:
	\[
	B(x) = 2, \quad C(x) = x^t, \quad D(x) = B(x)^2 + C(x)^2 = 4 + x^{2t}.
	\]
	If $t$ is odd, then $D(x)$ is irreducible by \cite[Ch. VI, Theorem 9.1]{Lang}. If $t$ is even, then
	\[
	x^{2t} + 4 = (x^t - 2x^{\frac{t}{2}} + 2)(x^t + 2x^{\frac{t}{2}} + 2),
	\]
	with both factors Eisenstein and thus irreducible.

	Let $T(k)$ be the $k$-th Catalan number, i.e., $T(k) = \frac{1}{k+1}\binom{2k}{k} \in \mathbb{Z}$.
	Construct the polynomial:
	\[
	A(x) = \sum_{i=0}^{2s} a(i) x^{i},
	\]
	where the coefficients $a(i)$ are given by:
	\[
	a(i) =
	\begin{cases}
		1 & \text{if } i = 2s,\\
		2 & \text{if } i = 2s - t \text{ or } i = 0,\\
		4 & \text{if } i = 2s - 2t \text{ and } i \ne 0,\\
		(-1)^{(j+1)/2} 4 T\left(\frac{j - 3}{2}\right) & \text{if } i = 2s - jt \text{ for some odd } j \ge 3,\\
		0 & \text{otherwise},
	\end{cases}
	\]
	so that 
	\[A(x) = x^{2s} + 2x^{2s-t} + 4x^{2s-2t} + 4x^{2s-3t} - 4x^{2s-5t} + 8x^{2s-7t} + \cdots+2.\]
	See Remark \ref{rk.Ax} for another description of $A(x)$. The polynomial $A(x)$ is Eisenstein and hence irreducible. 
	It follows from the factorization of $D(x)$ that $A(x)$ and $D(x)$ have no common roots. Observe that 
	\[
	A(0)D(0) = 8, \quad A(1)D(1)\text{ is odd},
	\]
	so there is no prime that divides $A(k)D(k) $ for all integers $k$. By  Granville's work, stated as Theorem~\ref{thm:Granville} above, 
	assuming the ABC conjecture, there are infinitely many positive integers $k$ such that $A(k)$ and $D(k)$ are square-free and coprime. 
	Moreover, $A(k)$ must be odd, since otherwise $k$ would be even, implying $4 \mid D(k)$ and contradicting the square-freeness of $D(k)$. Thus, $A(k), B(k), C(k)$, and $D(k)$ satisfy the conditions in Section~\ref{section:introcyclic}.
	
	Let $k$ be such an integer and define 
	\[
	K := \mathbb{Q}\bigl(\sqrt{A\,(D+2\sqrt{D})}\bigr),
	\]
	where $A = A(k)$ and $D = D(k)$. By Section~\ref{section:introcyclic}, we have
	\[D_K = c\,A^2 D^3,
	\quad \text{where } c \in \{1, 16, 64, 256\}.\]
	Hence,
	\[A^2 D^3 \;\le\; D_K \;\le\; 256 A^2 D^3.\]
	The leading term of $A(x)$ is $x^{2s}$. Thus, for any fixed $\epsilon > 0$, 
	\[(1-\epsilon)\,k^{2s} \;\le\; A \;\le\; (1+\epsilon)\,k^{2s}.\]
	Here and for the remainder of the proof, any inequality involving $k$ is assumed to hold for sufficiently large $k$. Similarly, since the leading term of $D(x)$ is $x^{2t}$, for any fixed $\delta > 0$,
	\[(1-\delta)\,k^{2t} \;\le\; D \;\le\; (1+\delta)\,k^{2t}.\]
	By choosing $\epsilon$ and $\delta$ sufficiently small, it follows that
	\[\tfrac{1}{2}\,k^{4s+6t} \;\le\; D_K \;\le\; 257\,k^{4s+6t}.\]

	Applying Proposition~\ref{Liouvillerealbound_cyc}, we obtain the lower bound
	\[\tfrac{1}{12593} D_K^{\frac{p}q} \leq \tfrac{1}{49}k^{2s+t} \leq M(\mathcal{O}_K).\]
	To derive an upper bound, consider \[\alpha_1 = \lfloor\rho\rfloor + k^s \sqrt{D} + \rho + \sigma,\]
	where 
	$\rho = \sqrt{A(D+2\sqrt{D})}$ and $\sigma = \sqrt{A(D-2\sqrt{D})}$. Let $\alpha_2, \alpha_3, \alpha_4$ be the other conjugates of $\alpha_1$ as given earlier.
	Then 
	\[\alpha_1 \leq 5k^{s+t}.\]

	We now show that $|k^s\sqrt{D} - \sigma| \leq 9.$
	Let $A_0(x) = A(x) - x^{2s}$ and $A_0 = A_0(k)$. 
	Then
	\begin{align*}
		\left|k^s \sqrt{D} - \sigma \right|
		= \left| \frac{k^{2s} D - A(D - 2\sqrt{D})}{k^s \sqrt{D} + \sigma} \right| 
		\le \left| \frac{k^{2s} D - A(D - 2\sqrt{D})}{k^{s+t}} \right|,
	\end{align*}
	where the inequality follows from $k^{s+t} \le k^s \sqrt{D} + \sigma$.  
	Factoring out $\sqrt{D}$ from the numerator in the last fraction and using $\sqrt{D} \le 2k^t$, we obtain
	\begin{align*}
		\left|k^s \sqrt{D} - \sigma \right|
		\le 2 \left| \frac{2A - (A - k^{2s})\sqrt{D}}{k^s} \right| 
		= 2 \left| \frac{4A^2 - A_0^2 D}{(2A + A_0 \sqrt{D}) k^s} \right|.
	\end{align*}
	Since $ 2k^{2s}\le 2A + A_0 \sqrt{D}$, it follows that
	\begin{align*}
		\left|k^s \sqrt{D} - \sigma \right|
		\le \left| \frac{4A^2 - A_0^2 D}{k^{3s}} \right| 
		= \left| \frac{4(k^{4s} + 2k^{2s} A_0) - A_0^2 k^{2t}}{k^{3s}} \right|=\left| \frac{f(k)}{k^{3s}}\right|
	\end{align*}
	where
	\[f(x) := \sum_{i=0}^{4s} f(i)x^i := 4x^{4s} + 8x^{2s}A_0(x) - A_0(x)^2x^{2t}.\]
	Consider the case when $\frac{p}{q} > \frac{3}{10}$, so that $s > t$. 
	In this case, the degree of $f(x)$ is less than $3s$. Suppose the degree is at least $2s$ and let $i \geq 3s$, which implies $i > 2s + t$. It is easy to see that $f(i) = 0$ if $i$ is not of the form $4s - jt$, or if $i = 4s - jt$ with $j = 0, 2$, or $j$ odd. When $j \geq 4$ and $j$ is even, we obtain
	\[
	f(4s - jt)
	= 16(-1)^{j/2} \left( -T\!\left(\tfrac{j-2}{2}\right) 
	+ \sum_{i=0}^{(j-4)/2} T(i)\,T\!\left(\tfrac{j-4}{2}-i\right) \right) 
	= 0,
	\]
	where the second equality follows from Segner's recurrence relation 
	(see, e.g., \cite[Eq.~(5.6)]{Koshy}). Since $f(i) = 0$ for $i \geq 3s$, the degree of $f(x)$ is less than $3s$.  As such, when $\tfrac{p}q> \tfrac3{10}$ the above calculation shows for sufficiently large $k$ 
	that $\left|k^s \sqrt{D} - \sigma \right| \leq 9$. 
	If $\frac{p}{q} = \frac{3}{10}$, setting $s = t = 1$, we find that $f(x) = 8x^3 + 12x^2$ and 
	we have
	\[
	\left| k^s \sqrt{D} - \sigma \right| 
	\leq \left| \frac{f(k)}{k^3} \right| = \left| \frac{8k^3 + 12k^2}{k^3} \right|.
	\]
	Hence,	$\left| k^s \sqrt{D} - \sigma \right| \leq 9$ when $\frac{p}{q} = \frac{3}{10}$ as well.
	
	We now bound the conjugates, recalling that from above, $|\alpha_1| \leq 5k^{s+t}$.
	Meanwhile,
	\begin{align*}
		|\rho - \sigma| &= \left| \sqrt{A(D + B\sqrt{D})} - \sqrt{A(D - B\sqrt{D})} \right| \le \frac{2AB\sqrt{D}}{k^{s + t}} \le 5 k^s.
	\end{align*}
	It follows that
	\[
	|\alpha_2| \le |\lfloor \rho \rfloor - \rho| + |k^s \sqrt{D} - \sigma| \le 10,
	\]
	\[
	|\alpha_3| \le |\lfloor \rho \rfloor - \rho| + |k^s \sqrt{D} - \sigma| \le 10,
	\]
	\[
	|\alpha_4| \le |\lfloor \rho \rfloor - \rho| + |k^s \sqrt{D} - \sigma| + 2|\rho - \sigma| \le 11 k^s.
	\]
	Therefore
	\[
	M(\mathcal{O}_K) \le M(\alpha_1) \le 5500 k^{2s + t} \le 11000 \:D_K^{\frac{p}q}.
	\]
	This completes the proof.
\end{proof}
\begin{remark}\label{rk.Ax}
	When $i \ne 0$ and is of the form $i = 2s - jt$ for some integer $j > 0$, the coefficient $a(i)$ of the polynomial $A(x)$ above can be uniformly expressed via the hypergeometric function:
	\[
	a(i) = 4 \cdot {}_2F_1(3-j, 2-j; 2; -1).
	\]
	For further details, see \cite[A198786]{oeis:A198786}. 
\end{remark}

\begin{prop}\label{realcyc_14310}
	Let $\tfrac{1}{4} \le \tfrac{p}{q} < \tfrac{3}{10}$ be a rational number. There are absolute  constants $c_1, c_2>0$ such that assuming the ABC conjecture there are infinitely many real cyclic quartic fields $K$ for which 
	\[ c_1 D_K^{\frac{p}{q}} \le M(\mathcal{O}_K) \le c_2 D_K^{\frac{p}{q}}.\]
\end{prop}
\begin{proof}
	We can write $\frac{p}{q} = \frac{2s + t}{4s + 6t}$ with $s, t \in \mathbb{N}_{\geq 1}$. For example, we may take $s = 6p - q$ and $t = 2q - 4p$.
	
	First, assume that $\frac{1}{4} < \frac{p}{q} < \frac{3}{10}$, so that $s < t < 2s$. Then there exists a positive integer $m$ such that
	\[
	\frac{8(m+1)-2}{32(m+1)-12} \le \frac{p}{q} < \frac{8m - 2}{32m - 12},
	\]
	which is equivalent to
	\[
	\frac{2(m+1)}{4(m+1)-2} \le \frac{s}{t} < \frac{2m}{4m - 2}.
	\]
	
	Let $r \equiv 5 \pmod{8}$ be a prime such that $r > 100m^2$. Such a prime can be expressed as $r = r_1^2 + r_2^2$, where $r_1, r_2 \in \mathbb{N}$, with $r_1$ odd and $r_2 \equiv 2 \pmod{4}$. Define the polynomials
	\[
	A(x) = x^{2s} + r_2 x^{2s - t} + r_2, \ \ 
	B(x) = r_2, \ \ 
	C(x) = r_1 + \sum_{i = 0}^{2m - 1} (-1)^i x^{t - i(2s - t)}, 
	\]
and let 	
	\[
	 D(x) = B(x)^2 + C(x)^2,
	\]
	where the inequality $\frac{s}{t} < \frac{2m}{4m - 2}$ ensures that each exponent of $x$ in the summation in $C(x)$ is positive. The polynomial $A(x)$ is Eisenstein and thus irreducible.
	We will now show that $D(x)$ is also irreducible. A direct computation shows
	\begin{align*}
		C(x)^2 ={}& \sum_{i = 0}^{2m - 1} (-1)^i(i + 1)x^{2t - i(2s - t)} + \sum_{i = 2m}^{4m - 2} (-1)^i(4m - 1 - i)x^{2t - i(2s - t)} \\
		&+ 2r_1 \sum_{i = 0}^{2m - 1} (-1)^i x^{t - i(2s - t)} + r_1^2.
	\end{align*}
	The sum of absolute values of the coefficients corresponding to positive-degree terms in $D(x)$ is at most
	\[
	F := \sum_{i = 0}^{2m - 1}(i + 1) + \sum_{i = 2m}^{4m - 2}(4m - 1 - i) + 4r_1 m = 4m(m + r_1),
	\]
	while the constant term is the prime $r$. If $r_1 > r_2$, then $r_1 > \frac{1}{2} \sqrt{r} > 5m$, and
	\[
	r - F > (r_1 - 2m)^2 - 8m^2 > 0.
	\]
	Likewise, if $r_2 > r_1$, then $r_2 > \frac{1}{2} \sqrt{r} > 5m$, and
	\[
	r - F > (r_2 - 2m)^2 - 8m^2 > 0.
	\]
	Hence, by \cite[Theorem~2.2.7]{Prasolov}, $D(x)$ is irreducible. Consequently, $A(x)D(x)$ has no repeated roots. Observe that the greatest common divisor of 
	\[
	A(0) D(0) = r_2 r \quad \text{and} \quad A(1) D(1) = (2r_2 + 1) r
	\]
	is the prime $r$. By Theorem~\ref{thm:Granville}, assuming the ABC conjecture, there exist infinitely many positive integers $k$ such that $A(k)$ and $D(k)$ are square-free and coprime.
	
	For such $k$, define $A = A(k)$, $B = B(k)$, $C = C(k)$, $D = D(k)$, and set
	\[
	K := \mathbb{Q} \left( \sqrt{A(D + B\sqrt{D})} \right).
	\]
	If $A(k)$ is odd, then $A, B, C, D$ satisfy the conditions in Section~\ref{section:introcyclic}. If $A(k)$ is even, then by \cite[Eq.~2.16]{zbMATH03995823},
	\[
	K = \mathbb{Q} \left( \sqrt{\frac{A}{2}(D + C\sqrt{D})} \right),
	\]
	and $\frac{A}{2}, C, B, D$ satisfy the conditions in Section~\ref{section:introcyclic}. In either case, by~\eqref{eq.disc_cyc}, the discriminant satisfies
	\[
	\tfrac{1}{2} k^{4s + 6t} \le D_K \le 257 k^{4s + 6t}.
	\]
	Here and for the remainder of the proof, any inequality involving $k$ is assumed to hold for sufficiently large $k$.
	Applying Proposition~\ref{Liouvillerealbound_cyc}, we derive the lower bound
	\[
	\tfrac{1}{24929} D_K^{\frac{p}q} \le \tfrac{1}{97} k^{2s + t} \le M(\mathcal{O}_K).
	\]
	
	Next we establish an upper bound. Consider
	\[
	\alpha_1 = \lfloor \rho \rfloor + k^s \sqrt{D} + \rho + \sigma,
	\]
	where 
	\[
	\rho = \sqrt{A(D + B\sqrt{D})},\quad \sigma = \sqrt{A(D - B\sqrt{D})}.
	\]
	Clearly, $\alpha_1 \in \mathcal{O}_K$ and $K = \mathbb{Q}(\alpha_1)$. Let $\alpha_2, \alpha_3, \alpha_4$ denote the other conjugates of $\alpha_1$ as given earlier. We find
	\[
	|\alpha_1| \le 5 k^{s + t}.
	\]
	We now show that $|k^s\sqrt{D} - \sigma| \leq 9$. Define $A_0(x) := A(x) - x^{2s}$, and set $A_0 := A_0(k)$. We have
	\begin{align*}
		\left| k^s \sqrt{D} - \sigma \right|
		= \left| \frac{k^{2s}D - A(D - r_2 \sqrt{D})}{k^s \sqrt{D} + \sigma} \right| 
		\leq \left| \frac{k^{2s}D - A(D - r_2 \sqrt{D})}{k^{s+t}} \right|,
	\end{align*}
	where the inequality follows from $k^{s+t} \leq k^s \sqrt{D} + \sigma$.  
	Factoring out $\sqrt{D}$ from the numerator in the last fraction and using $\sqrt{D} \leq 2k^t$, we obtain
	\begin{align*}
		\left| k^s \sqrt{D} - \sigma \right|
		\leq 2 \left| \frac{r_2 A - (A - k^{2s})\sqrt{D}}{k^s} \right| 
		= 2 \left| \frac{r_2^2 A^2 - A_0^2 D}{(r_2 A + A_0 \sqrt{D}) \, k^s} \right|.
	\end{align*}
	Since $r_2 A + A_0 \sqrt{D} \geq 2k^{2s}$, it follows that
	\begin{align*}
		\left| k^s \sqrt{D} - \sigma \right|
		\leq \left| \frac{r_2^2 A^2 - A_0^2 D}{k^{3s}} \right| 
		= \left| \frac{r_2^2 (k^{4s} + 2k^{2s}A_0) - A_0^2 C^2}{k^{3s}} \right| =\left| \frac{f(k)}{k^{3s}}\right|
	\end{align*}
	where $f(x) := r_2^2 (x^{4s} + 2 x^{2s} A_0(x)) - A_0(x)^2 C(x)^2$. Expanding $f(x)$ yields
	\begin{align*}
		f(x) 
		={}&\, 2r_2^2 x^{2t - (2m - 2)(2s - t)} + 2r_1 r_2^2 x^{t - (2m - 2)(2s - t)}- r_2^2 x^{2t - (4m - 2)(2s - t)}\\
		&  + 2r_1 r_2^2 x^{t - (2m - 1)(2s - t)} - 2r_1 r_2^2 x^{4s - t} + 2r_2^3 x^{4s - t} - 2r_1 r_2^2 x^{2s} \\
		&+ 2r_2^3 x^{2s} - r_1^2 r_2^2 x^{4s - 2t} - 2r_1^2 r_2^2 x^{2s - t} - r_1^2 r_2^2.
	\end{align*}
	If $\tfrac{s}{t} \ne \tfrac{2}{3}$, we claim that the degree of $f(x)$ is less than $3s$. First, observe that the largest exponent among the first four terms in the formula for $f(x)$ above is \[2t - (2m - 2)(2s - t).\] The conditions \[\frac{2(m+1)}{4(m+1)-2} \le \frac{s}{t} \quad \text{and} \quad \frac{s}{t} \ne \frac{2}{3}\] together imply that \[2t - (2m - 2)(2s - t) < 3s.\] Hence the exponent of each of the first four terms is less than $3s$. Moreover, as $t > s$, it follows that the exponent of every remaining term is also less than $3s$. This proves the claim. Therefore, for $k$ sufficiently large, when$\tfrac{s}{t} \ne \tfrac{2}{3}$ we have  $| k^s \sqrt{D} - \sigma | \le 9.$

	When $\tfrac{s}{t} = \tfrac{2}{3}$, we have $m = 1$, and the leading term of $f(x)$ is $2r_2^2 x^{3s}$. Choosing $r = 173$ yields $r_1 = 13$ and $r_2 = 2$, so the leading term becomes $8x^{3s}$. Because  $| k^s \sqrt{D} - \sigma | \le | \frac{f(k)}{k^{3s}} |,$ we obtain $| k^s \sqrt{D} - \sigma| \le 9.$

	Additionally, 	
	\[
	|\rho - \sigma| = \left| \frac{2AB\sqrt{D}}{\sqrt{A(D + B\sqrt{D})} + \sqrt{A(D - B\sqrt{D})}} \right| 
	\le \frac{2AB\sqrt{D}}{k^{s+t}} 
	= 2r_2 \frac{A\sqrt{D}}{k^{s+t}}.
	\]
	It follows that
	\[
	|\alpha_2| \le |\lfloor \rho \rfloor - \rho| + |k^s \sqrt{D} - \sigma| \le 10,
	\]
	\[
	|\alpha_3| \le |\lfloor \rho \rfloor - \rho| + |k^s \sqrt{D} - \sigma| \le 10,
	\]
	\[
	|\alpha_4| \le |\lfloor \rho \rfloor - \rho| + |k^s \sqrt{D} - \sigma| + 2|\rho - \sigma| \le 4\sqrt{r} k^s.
	\]
	Therefore,
	\[
	M(\mathcal{O}_K) \le M(\alpha_1) \le 2000\sqrt{r} k^{2s + t} \le 4000\sqrt{r} D_K^{\frac{p}q}.
	\]
	
	Finally, if $\frac{p}{q} = \frac{1}{4}$, then $2s = t$, and we define
	\[
		A(x) = x^{2s} + 2, \ \  B(x) = 2, \ \ 
		C(x) = x^{t},  \ \ D(x) = B(x)^2 + C(x)^2.
	\]
	The analysis in this case proceeds analogously to the proof of Proposition~\ref{realcyc_31012}, and we omit further details.
\end{proof}

\begin{remark}
	It appears that the approach used in the proofs of Propositions~\ref{realcyc_31012} and~\ref{realcyc_14310} does not yield families of real cyclic quartic fields that give exponents lying in the interval $\left( \frac{1}{6}, \frac{1}{4} \right)$. We now explain the difficulties.
	
	Assume that $\frac{1}{6} < \frac{p}{q} < \frac{1}{4}$. We may write
	\[
	\frac{p}{q} = \frac{2s + t}{4s + 6t}
	\quad \text{with} \quad s, t \in \mathbb{N}_{\geq 1}.
	\]
	Let $A(x), B(x), C(x), D(x) \in \mathbb{Z}[x]$ be such that $\deg A(x) = 2s$, $\deg D(x) = 2t$, and
	\[
	B(x)^2 + C(x)^2 = D(x).
	\]
	Suppose there exist infinitely many positive integers $k$ such that $A(k), B(k), C(k), D(k)$ satisfy the conditions in Section~\ref{section:introcyclic}, with $A(k) > 0$. For such $k$, consider the associated family of real cyclic quartic fields
	\[
	K := \mathbb{Q} \left( \sqrt{A(D + B\sqrt{D})} \right),
	\]
	where $A = A(k)$, $B = B(k)$, $C = C(k)$, and $D = D(k)$. Let $\rho$ and $\sigma$ be defined as in~\eqref{rho_sig}.
	
	Following the approach in Propositions~\ref{realcyc_31012} and~\ref{realcyc_14310}, we consider the element
	\[
	\alpha_1 = \lfloor \rho \rfloor + k^s \sqrt{D} + \rho + \sigma.
	\]
	We aim to obtain that
	\[
	\left| k^s \sqrt{D} - \sigma \right| \le c_1
	\]
	for some constant $c_1 > 0$, and use this to deduce the bound
	\begin{equation}\label{eq.upbd}
		M(\alpha_1) \le c_2 k^{2s + t}
	\end{equation}
	for some constant $c_2 > 0$.
	
	\medskip
	
	\textbf{Case 1: $\deg B(x) = r > 0$.} In this case, we estimate
	\[
	|\alpha_1| \ge \tfrac{1}{2} k^{s + t},
	\]
	and
	\[
	|\rho - \sigma| = 2\sqrt{A} B \cdot \frac{\sqrt{D}}{\sqrt{D + B\sqrt{D}} + \sqrt{D - B\sqrt{D}}}
	\ge \tfrac{1}{2} k^{s + r}.
	\]
	Assume that $\left| k^s \sqrt{D} - \sigma \right| \le c_1$ for some constant $c_1$. Let $\alpha_4$ be the conjugate of $\alpha_1$ as given earlier. Then we derive
	\begin{align*}
		|\alpha_4| &= \left| 2(\rho - \sigma) + (\sigma - k^s \sqrt{D}) - (\rho - \lfloor \rho \rfloor) \right| \\
		&\ge 2|\rho - \sigma| - c_1 - 1 \ge \tfrac{1}{2} k^{s + r}.
	\end{align*}
	Hence,
	\[
	M(\alpha_1) \ge \tfrac{1}{4} k^{2s + t + r},
	\]
	which contradicts the upper bound in~\eqref{eq.upbd}.
	
	\medskip
	
	\textbf{Case 2: $B(x)$ is constant.} Define $A_0(x) := A(x) - x^{2s}$, and set $A_0 := A_0(n)$. Applying computations analogous to those in the proofs of Theorem~\ref{realcyc_14310}, we find
	\[
	\left| k^s \sqrt{D} - \sigma \right| \le 2 \left| \frac{B^2(k^{4s} + 2k^{2s} A_0) - A_0^2 k^{2t}}{k^{3s}} \right|.
	\]
	Since $\frac{1}{6} < \frac{p}{q} < \frac{1}{4}$, we have $s > 0$ and $2s < t$. Then the degree of the numerator polynomial
	\[
	B(x)^2 (x^{4s} + 2x^{2s} A_0(x)) - A_0(x)^2 x^{2t}
	\]
	is $2t + \deg A_0(x) > 4s$, so the right-hand side does not remain bounded as $k \to \infty$. Therefore, no constant bound is attained on $\big| k^s \sqrt{D} - \sigma \big|$.

	In both cases, this approach does not work to construct the desired family of real cyclic quartic fields with exponents in $\left( \frac{1}{6}, \frac{1}{4} \right)$.
\end{remark}

\subsection{Experimental Data for Real Cyclic Quartics}
Here we present our results on computing $M(\mathcal{O}_K)$ over all real cyclic quartic fields with bounded discriminant. The numerical data for $|D_K| \leq \bd$ appear in Figure~\ref{fig.realcyc1}. To explain our approach, we begin with a  proposition.

\begin{figure}[h!]
	\begin{center}
		\includegraphics[scale=0.6]{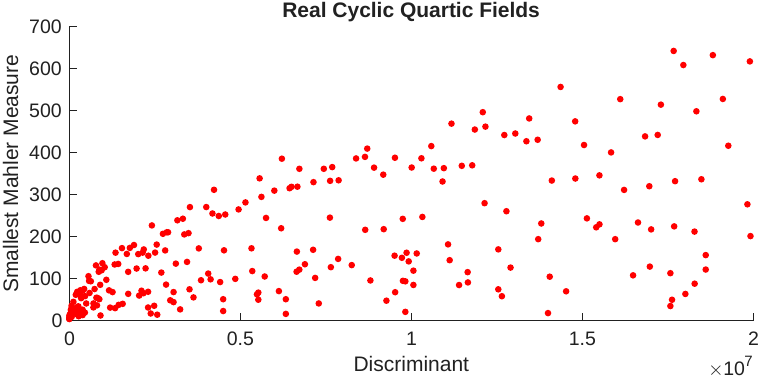} 
		\caption{$M(\mathcal{O}_K)$ for real cyclic quartic fields $K$ with $D_K \leq \bd$}
		\label{fig.realcyc1}
	\end{center}
\end{figure}

\begin{prop}\label{prop.srch}
	Let $A, B, C, D \in \mathbb{Z}$ satisfy the conditions in Section~\ref{section:introcyclic}, with $A > 0$, and define
	\[
	K = \mathbb{Q}\big(\sqrt{A(D + B\sqrt{D})}\big).
	\]
	Let $\rho$ and $\sigma$ be as defined in equation~\eqref{rho_sig}, and let $x_1, x_2, x_3, x_4 \in \mathbb{Z}$. Consider the element
	\[
	\alpha_1 = \tfrac{1}{4} \Big( x_1 + x_2 \sqrt{D} + x_3 \rho + x_4 \sigma \Big).
	\]
	Assume $\alpha_1 \in \mathcal{O}_K$, and let $L \in \mathbb{R}$. If $M(\alpha_1) \leq L$, then
	\[
	|x_1| \leq 4L, \quad |x_2| \leq \frac{4L}{\sqrt{D}}, \quad |x_3| \leq \frac{4L}{\rho}, \quad |x_4| \leq \frac{4L}{\sigma}.
	\]
\end{prop}

\begin{proof}
	Let $\alpha_2, \alpha_3, \alpha_4$ be the other conjugates of $\alpha_1$, as given earlier. By choosing an appropriate conjugate of $\alpha_1$, we may assume that $x_1 x_2 \geq 0$. Without loss of generality, we can further assume that $|\alpha_1| \geq \frac{1}{4} |x_1|$, as justified below.
	
	Consider first the case where $x_1 \geq 0$ and $x_2 \geq 0$. If $|\alpha_1| < \frac{1}{4} x_1$, then $\alpha_1 < \frac{1}{4} x_1$, which implies
	\[
	\tfrac{1}{4} \big( x_2 \sqrt{D} + x_3 \rho + x_4 \sigma \big) < 0.
	\]
	This yields
	\[
	-\alpha_3 = \tfrac{1}{4} (x_2 \sqrt{D} + x_3 \rho + x_4 \sigma) - \tfrac{1}{4}(x_1 + 2x_2) < -\tfrac{1}{4} x_1.
	\]
	We may then replace $\alpha_1$ with $-\alpha_3$, which satisfies our assumptions and preserves the Mahler measure. The case where $x_1, x_2 \leq 0$ is similar and omitted.
	
	Now, if $|x_1| > 4L$, then $|\alpha_1| > L$, and hence $M(\alpha_1) > L$, contradicting our assumption. This establishes the first inequality. The remaining statements follow similarly.
\end{proof}

\begin{figure}[h!]
	\begin{center}
		\includegraphics[scale=0.6]{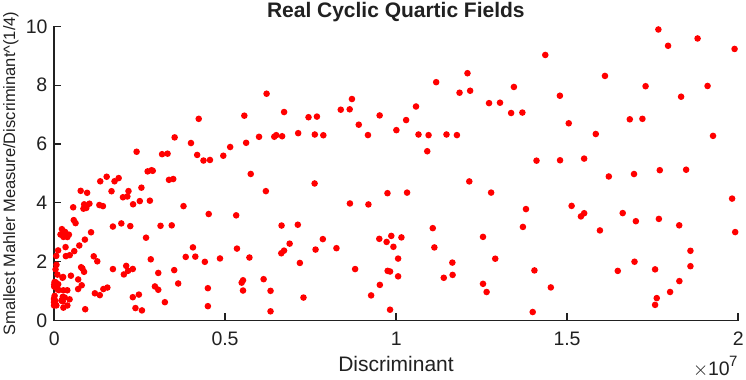} 
		\includegraphics[scale=0.6]{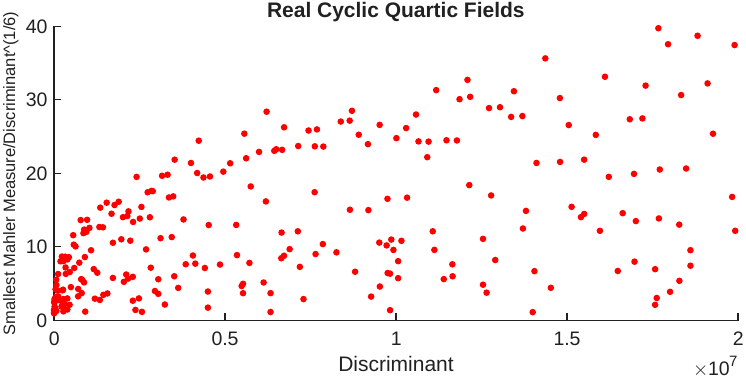} 
		\caption{For real cyclic quartic fields $K$ with $|D_K| \leq \bd$, the figures show  $M(\mathcal{O}_K)(D_K)^{-\frac14}$, and $M(\mathcal{O}_K)(D_K)^{-\frac16}$, respectively.}
		\label{fig.realcyc2}
	\end{center}
\end{figure}

We now describe our computational method. Let $K$ be a real cyclic quartic field. Then there exist unique integers $A, B, C, D$ with $A > 0$ satisfying the conditions in Section~\ref{section:introcyclic} such that
\[
K = \mathbb{Q}\big(\sqrt{A(D + B\sqrt{D})}\big).
\]
Using equation~\eqref{eq.disc_cyc}, we have the inequality $A^2 D^3 \leq D_K$. Let $E$ be a positive real number. If $D_K \leq E$, then it follows that $A \leq \sqrt{E}$ and $D \leq \sqrt[3]{E}$. This allows us to enumerate all real cyclic quartic fields with discriminant at most $E$.

For each such field, we first apply Proposition~\ref{prop.cycint} to find the smallest Mahler measure among integral generators with
\[
0 \leq x_1, x_2, x_3, x_4 \leq A + D,
\]
where the bound is heuristically chosen. We then use this Mahler measure as the value of $L$ in Proposition~\ref{prop.srch} to identify all integral generators that could attain the minimal Mahler measure. We compute over this list to determine the true minimal Mahler measure $M(\mathcal{O}_K)$. The results for $D_K \leq \bd$ are presented in Figure~\ref{fig.realcyc1}. Theorem~\ref{realcyc_main} does not address the exponents in the range $(\tfrac16,\tfrac14)$ and  we present Figure~\ref{fig.realcyc2} to better understand these exponents.

\section{Imaginary Cyclic Quartics}\label{section:ccyclic}

Let $K$ be an imaginary cyclic quartic number field. 
As in Section~\ref{section:introcyclic} we can write $K=(\sqrt{A(D+B\sqrt{D})})$ where $A,B,C,$ and $D$ are rational integers which satisfy $A<0$ and $D=B^2+C^2$ with $B>0$. 
The theoretical bounds for imaginary cyclic quartics are given by (**) and are
\[
2^{-\frac{12}5} D_K^{\frac15} \leq M(\Ox_K) \leq D_K
\]
unless $Tor(K^{\times})\neq \{\pm 1\}$. The only imaginary cyclic quartic number fields where $Tor(K^{\times})\neq \{\pm 1\}$ are the splitting fields of $\Phi_5= x^4+x^2+x+1$ and $\Phi_{10}=x^4-x^3+x^2-x+1$. These cyclotomics satisfy $M(\Ox_K)=1$. 
Let 
\[
\rho = \sqrt{A(D + B\sqrt{D})}, \qquad 
\sigma = \sqrt{A(D - B\sqrt{D})},
\]
and define
\[
T = \Bigl\{\, 
\tfrac{1}{4}\bigl(x_1 + x_2\sqrt{D} + x_3\rho + x_4\sigma \bigr)
\;:\;
x_i \in \mathbb{Z} 
\Bigr\}.
\]
By Section~\ref{section:introcyclic}, we know that $\mathcal{O}_K \subset T$. Consequently, 
$$\min_{\alpha\in T}\{M'(\alpha) \;:\; \mathbb{Q}(\alpha)=K\}\le M(\mathcal{O}_K),$$
where $M'$ is given as in the equation \eqref{eq.M'}.
Here we use $M'$ rather than $M$, since elements of $T$ may not be integral.
Our next goal is to obtain a lower bound for $M(\mathcal{O}_K)$ by deducing a lower bound for the left-hand side of the last inequality. 
Take $\alpha = \alpha_1 \in T$. Then there exist $x_1,x_2,x_3,x_4\in \mathbb{Z}$ such that
\[
\alpha_1 = \tfrac{1}{4}\bigl(x_1 + x_2\sqrt{D} + x_3\rho + x_4\sigma \bigr).
\]
The other conjugates of $\alpha_1$ are given by
\[
\begin{aligned}
	\alpha_2 &= \tfrac{1}{4}\bigl(x_1 - x_2\sqrt{D} - x_4\rho + x_3\sigma\bigr),\\
	\alpha_3 &= \tfrac{1}{4}\bigl(x_1 + x_2\sqrt{D} - x_3\rho - x_4\sigma\bigr),\\
	\alpha_4 &= \tfrac{1}{4}\bigl(x_1 - x_2\sqrt{D} + x_4\rho - x_3\sigma\bigr).
\end{aligned}
\]
	One computes
	\[
	\lvert\alpha_1\rvert^2 = \alpha_1\alpha_3
	= \tfrac1{16}\bigl[(x_1 + x_2\sqrt{D})^2 + (x_3\sqrt{\lvert A\rvert(D+B\sqrt{D})} + x_4\sqrt{\lvert A\rvert(D-B\sqrt{D})})^2\bigr],
	\]
	and similarly
	\[
	\lvert\alpha_2\rvert^2 = \alpha_2\alpha_4
	= \tfrac1{16}\bigl[(x_1 - x_2\sqrt{D})^2 + (x_4\sqrt{\lvert A\rvert(D+B\sqrt{D})} - x_3\sqrt{\lvert A\rvert(D-B\sqrt{D})})^2\bigr].
	\]

	We represent $\alpha$ by the tuple $(x_1, x_2, x_3, x_4)$. By considering conjugates if necessary, it suffices to minimize $M'(\alpha)$ over the three families:
	\begin{equation}\label{em.3cases} \tag{$\dagger$}
	(x_1, x_2, 1, 0), \quad (x_1, x_2, 0, 1), \quad (0, 0, x_3, x_4),
	\end{equation}
	with $x_3, x_4 > 0$. 

We now improve the exponent $\tfrac15$ on the left‐hand side of $(*)$ and $(**)$  to $\tfrac13$.
\begin{theorem}\label{thm_cyccpxlb}
	Let $K$ be an imaginary cyclic quartic field. Then
	\[
	\tfrac{1}{128} D_K^{\frac13} \;\le\; M(\mathcal{O}_K).
	\]
\end{theorem}
\begin{proof}
	Examining the three families in (\ref{em.3cases}), one sees that in each case
	\[
	M'(\alpha)\;\ge\;|\alpha_i|^2
	\;\ge\;
	\tfrac1{16} |A|\,D
	\]
	for $i=1$ or $2$.  Hence
	\[
	M(\mathcal{O}_K)
	\ge\;
	\tfrac1{16} |A|\,D
	\;\ge\;
	\tfrac1{128} |D_K|^{\frac13}.
	\]
\end{proof}

\begin{prop}\label{Liouvillecpxbound_cyc}
	Let $A,B,C,D$ be integers satisfying the conditions in Section~\ref{section:introcyclic}  with $A<0$, and set
	\[
	K=\mathbb{Q}\bigl(\sqrt{A\,(D+B\sqrt{D})}\bigr).
	\]
	Then
	\[
	\tfrac{1}{2304}A^2\,D \; \le \; M\bigl(\mathcal{O}_K\bigr).
	\]
\end{prop}

\begin{proof}
		It suffices to consider the three cases described in~(\ref{em.3cases}). If $\alpha_1$ is represented by the tuple $(x_1,x_2,1,0)$ or $(x_1,x_2,0,1)$, then
	\[
	M'(\alpha_1)
	= \frac{\lvert A\rvert\,(D+B\sqrt{D})\;\lvert A\rvert\,(D-B\sqrt{D})}{256}
	= \frac{A^2\,D\,C^2}{256}
	\geq
	\frac{A^2\,D}{256}.
	\]

	Next, suppose $\alpha_1=(0,0,x_3,x_4)$ with $x_3,x_4>0$. Set
	\[
	\lambda := \lvert\alpha_2\rvert
	= \frac{1}{4}\Bigl\lvert
	x_4\sqrt{\lvert A\rvert\,(D+B\sqrt{D})}
	- x_3\sqrt{\lvert A\rvert\,(D-B\sqrt{D})}
	\Bigr\rvert.
	\]
	Then
	\[
	\biggl\lvert\frac{x_4}{x_3} - \frac{\sqrt{D}-B}{C}\biggr\rvert
	= \frac{4\lambda}{x_3\sqrt{\lvert A\rvert\,(D+B\sqrt{D})}}.
	\]
	Let 
	\[
	\mu = \mu\!\Bigl(\tfrac{\sqrt{D}-B}{C}\Bigr)
	\geq \frac{1}{C+2\sqrt{D}}
	\]
	as in Theorem~\ref{thm.Liou}. We have
	\[
	x_3
	\geq
	\frac{\mu\sqrt{\lvert A\rvert\,(D+B\sqrt{D})}}{4\lambda}
	\geq
	\frac{\sqrt{\lvert A\rvert\,(D+B\sqrt{D})}}
	{4\lambda\,(C+2\sqrt{D})}.
	\]
	Hence
	\[
	\lvert\alpha_1\rvert^2
	\geq
	\frac{x_3^2\,\lvert A\rvert\,(D+B\sqrt{D})}{16}
	\geq
	\frac{A^2\,(D+B\sqrt{D})^2}
	{256\,\lambda^2\,(C+2\sqrt{D})^2}.
	\]

	Since $M'(\alpha_1)\geq\lvert\alpha_1\rvert^2\lvert\alpha_2\rvert^2=\lvert\alpha_1\rvert^2\lambda^2$, we deduce
	\[
	M'(\alpha_1)
	\geq
	\frac{A^2\,(D+B\sqrt{D})^2}
	{256\,(C+2\sqrt{D})^2}.
	\]
	Therefore, we conclude
	\[
	M(\mathcal{O}_K)
	\geq
	\frac{A^2\,D\,(B+\sqrt{D})^2}
	{256\,(C+2\sqrt{D})^2}
	\geq
	\frac{1}{2304}A^2\,D,
	\]
	where the last inequality follows from $B\geq 0$ and $C\leq \sqrt{D}$.

\end{proof}

We now improve the exponent $\tfrac15$ on the left‐hand side of $(*)$ and $(**)$  to $\tfrac13$.
\begin{cor}
	Let $K=\mathbb{Q}(\sqrt{A(D+B\sqrt{D})})$ be an imaginary cyclic quartic field, where $A,B,C,$ and $D$ with $A<0$ satisfy the conditions given in Section~\ref{section:introcyclic}.  Then
	\[
	\tfrac{1}{14630} |A|^\frac43 \,D_K^\frac13\;\le\; M(\mathcal{O}_K).
	\]
\end{cor}
\begin{proof}
By Proposition~\ref{Liouvillecpxbound_cyc}, 
	$$M(\mathcal{O}_K)\geq\tfrac{1}{2304}A^2\,D\geq  \tfrac{1}{14630 } |A|^\frac43 (256A^2D^3)^\frac13.$$\end{proof}

We now prove Theorem~\ref{thm:condimg_cyc} by showing that assuming the ABC conjecture, all rational exponents in $[\tfrac13, 1]$ are realized by infinitely many imaginary cyclic quartic fields. 
	
\begin{proof}[proof of Theorem~\ref{thm:condimg_cyc}]
	Write
	\[
	\frac{p}{q}
	\;=\;
	\frac{s+t}{s+3t},
	\qquad
	s,t\in\mathbb{Z}_{\ge0},
	\]
	for example by taking $s=3p-q$ and $t=q-p$. Define the polynomials
	\[
	A(x)=-(x^s+2),
	\qquad
	D(x)=x^{2t}+1.
	\]
	Since $A(x)D(x)$ has no repeated roots and $A(1)D(1)=-6$ is square-free, Theorem~\ref{thm:Granville} implies, under the ABC conjecture, that there are infinitely many positive integers $k$ for which $A(k)$ and $D(k)$ are both square-free and coprime.
	For such an $k$, set
	\[
	A = A(k),\quad
	B = k^t,\quad
	C = 1, \quad
	D = D(k),
	\]
	and let 
	\[\rho = \sqrt{\,A\bigl(D + B\sqrt{D}\bigr)\,}\quad \text{and}\quad K=\mathbb{Q}(\rho).\]
	If $k$ is odd, then $A, B, C, D$ satisfy the conditions in Section~\ref{section:introcyclic}. If $k$ is even, then, from \cite[Eq.~(2.16)]{zbMATH03995823}
	\[
	K= \mathbb{Q} \bigl( \sqrt{\frac{A}{2} \, (D + C \sqrt{D})} \bigr),
	\]
	and $\frac{A}{2}, C, B, D$ satisfy the conditions in Section~\ref{section:introcyclic}. In either case, by \eqref{eq.disc_cyc}, the discriminant of $K$ satisfies
	\[
	\tfrac{1}{4}k^{2s+6t} \;\le\; D_K \;\le\; 257\,k^{2s+6t}.
	\]
	Here and for the remainder of the proof, any inequality involving $k$ is assumed to hold for sufficiently large $k$.
	By Proposition~\ref{Liouvillecpxbound_cyc} we obtain
	\[
	M\bigl(\mathcal{O}_K\bigr)
	\;\ge\;
	\frac{A^2D}{9216}
	\;\ge\;
	\frac{k^{2(s+t)}}{9216}
	\;\ge\;
	\frac{D_K^{\frac{p}{q}}}{2368512},
	\]
	while
	\[
	M\bigl(\mathcal{O}_K\bigr)
	\;\le\;
	M(\rho)
	\;=\;
	A^2D
	\;\le\;
	2\,k^{2(s+t)}
	\;\le\;
	8\,D_K^{\frac{p}{q}}.
	\]
	This completes the proof.
\end{proof}

\begin{cor}\label{cor:ccyclic13} 
 There are infinitely many positive integers $k$ such that $k^2+1$ is square-free. For large enough such $k$, the fields $K_k=\mathbb{Q}\Big(\sqrt{-(k^2+1+k\sqrt{k^2+1})}\Big)$ satisfy
    \begin{equation*}
       \tfrac{1}{14649}D_{K_k}^\frac{1}{3}\leq M(\mathcal{O}_{K_k})\leq 2D_{K_k}^\frac{1}{3}.
    \end{equation*}
    \end{cor}
\begin{proof}
    Since $\deg(x^2+1)=2$, there are infinitely many positive integers $k$ such that $k^2+1$ is square-free. Noticing that $A=-1$ is odd and proceeding as in the proof Theorem \ref{thm:condimg_cyc} with $s=0$ and $t=1$ completes the proof.
\end{proof}
\begin{cor}\label{cor:ccyclic12} 
 There are infinitely many positive integers $k$ such that $k(k^2+1)$ is square-free. For large enough such $k$, the fields $K_k=\mathbb{Q}\Big(\sqrt{-k(k^2+1+k\sqrt{k^2+1})}\Big)$ satisfy
    \begin{equation*}
       \tfrac{1}{147744}D_{K_k}^\frac{1}{2}\leq M(\mathcal{O}_{K_k})\leq 8D_{K_k}^\frac{1}{2}.
    \end{equation*}
\end{cor}
\begin{proof}
    Since $\deg(k(k^2+1))=3$, by Hooley \cite{MR214556} there are infinitely many integers $k$ such that $k(k^2+1)$ is square-free. Proceeding as in the proof Theorem \ref{thm:condimg_cyc} with $s=1$ and $t=1$ completes the argument.
\end{proof}
\begin{cor}\label{cor:ccyclic1} 
 For any square-free integer $k>1$ with $\gcd(k,2)=1$, the fields $K_k=\mathbb{Q}\Big(\sqrt{-k(2+\sqrt{2})}\Big)$ satisfy
    \begin{equation*}
       \tfrac{1}{2359296}D_{K_k}\leq M(\mathcal{O}_{K_k})\leq \tfrac14D_{K_k}.
    \end{equation*}
\end{cor}
\begin{proof}
We have $D_{K_k}=2048k^2$ and $A=-k$ is odd.  Following the approach used in the proof of  Theorem \ref{thm:condimg_cyc} with $s=1$ and $t=0$ completes the proof.
\end{proof}

\bibliographystyle{amsplain}
\bibliography{small.bib}

\noindent\rule{4cm}{.5pt}
\vspace{.25cm}

\noindent {\small Bishnu Paudel}\\
{\small Department of Mathematics and Statistics, University of Minnesota Duluth, Duluth MN 55812} \\
{\small email: {\tt bpaudel@umn.edu }}

\vspace{.25cm}

\noindent {\small Kathleen Petersen}\\
{\small Department of Mathematics and Statistics, University of Minnesota Duluth, Duluth MN 55812} \\
{\small email: {\tt kpete@umn.edu}}

\vspace{.25cm}

\noindent {\small Haiyang  Wang}\\
{\small Department of Mathematics and Statistics, University of Minnesota Duluth, Duluth MN 55812} \\
{\small email: {\tt wan02600@umn.edu}}

\end{document}